\documentclass[smallcondensed]{svjour3}     
\smartqed  
\usepackage{graphicx}
\usepackage{amsfonts}
\usepackage[applemac]{inputenc}
\usepackage{amssymb}
\usepackage{makeidx}
\usepackage{graphicx}
\usepackage{amsmath}
\usepackage[displaymath]{lineno}
\usepackage[singlespacing]{setspace}
\usepackage[normalem]{ulem}
\usepackage{lineno}
\usepackage{color}%

\newtheorem{theo}{Theorem}
\newtheorem{lem}[theo]{Lemma}
\newtheorem{prop}[theo]{Proposition}
\newtheorem{cor}[theo]{Corollary}
\newtheorem{rem}{Remark}

\input{tcilatex}
\begin{document}

\title{Subdifferential of the supremum function\thanks{Research supported by CONICYT
(Fondecyt 1190012 and 1190110), Proyecto/Grant PIA AFB-170001, MICIU of Spain
and Universidad de Alicante (Grant Beatriz Galindo BEAGAL 18/00205), and
Research Project PGC2018-097960-B-C21 from MICINN, Spain. The
research of the third author is also supported by the Australian ARC
- Discovery Projects DP 180100602.}}

\subtitle{Moving back and forth between continuous and non-continuous settings}


\author{R. Correa         \and
     A. Hantoute \and   M. A. L\'opez 
}


\institute{R. Correa \at
             Universidad de O'Higgins, Chile, and DIM-CMM of
Universidad de Chile \\
              \email{rcorrea@dim.uchile.cl}           
           \and
           A. Hantoute \at
              Center for Mathematical Modeling (CMM), Universidad de
Chile, and Universidad de Alicante, Spain\\
              \email{ahantoute@dim.uchile.cl (corresponding author)}
               \and
           M. A. L\'opez \at
              Universidad de Alicante, Spain, and CIAO, Federation University, Ballarat, Australia\\
              \email{marco.antonio@ua.es}
}

\date{Received: date / Accepted: date}

\titlerunning{}\maketitle

\begin{abstract}
In this paper we develop general formulas for the subdifferential of the pointwise supremum of convex functions, which cover and unify both\ the compact
continuous and the non-compact non-continuous settings. From the non-continuous to the continuous setting, we proceed by a compactification-based approach which leads us to problems having compact index sets and upper semi-continuously indexed mappings, giving rise to new characterizations of the subdifferential of the supremum by means of upper semicontinuous regularized functions and an enlarged compact index set. In the opposite sense, we rewrite the subdifferential of these new regularized functions by using the original data, also leading us to new results on the subdifferential of the supremum. We give two applications in the last section, the first one concerning the nonconvex Fenchel duality, and the second one establishing Fritz-John and KKT conditions in convex semi-infinite programming.


\end{abstract}

\section{Introduction}

In this paper we deal with the characterization of the
subdifferential of the pointwise supremum $f:=\sup_{t\in T}f_{t}$ of a
family of convex functions $f_{t}:X\rightarrow\mathbb{R}\cup\{\pm\infty
\}$, $t\in T$, with $T$ being an arbitrary nonempty set,
defined on a separated locally convex space $X.$ We obtain new
characterizations which allow us to unify both the \emph{compact continuous}  and the \emph{non-compact non-continuous} setting (\cite{CHL19}, \cite{CHL19b}, \cite{Ps65}, \cite{Va69},
etc.). The first setting relies on the following standard conditions in the literature of
convex analysis and non-differentiable semi-infinite programming:
\[
T \text{ is compact and the mappings } f_{(\cdot)}(z),\text{ }z\in X, \text{ are
upper semi-continuous.}
\]
In the other framework, called the \emph{non-compact non-continuous} setting, we do not assume the above conditions. In other words (see, i.e., \cite{HL08}, \cite{HLZ08}, \cite{LiNg11}, \cite{Io12}, \cite{So01}, \cite{Va69}, \cite{Vo94}, etc.):
\begin{gather*}
T \text{ is an arbitrary set, possibly infinite and without any prescribed
topology}, \\
\text{and no requirement is imposed on the mappings }f_{(\cdot)}(z).
\end{gather*}

Going from the non-continuous to the continuous setting,
we follow an approach based on the Stone-\v{C}ech compactification of the
index set $T$. At the same time, we build an appropriate enlargement of the
original family $f_{t},\ t\in T,$ which ensures the
fulfillment of the upper semi-continuity property required in the compact
setting. Since the new setting is naturally compact, by applying the results
in \cite{CHL19,CHL19b}, we
obtain new characterizations given in terms of the exact
subdifferential at the reference point of the new functions and the extended
active set. In this way, we succeed in unifying both settings. In
\cite{CHL19c}, we gave the first steps in this direction, using compactification arguments, but in the current paper we go further into the subject with some enhanced formulas.

To move in the other direction, we rewrite the subdifferential of
these new regularizing functions in terms of the original data,
and this also leads us to new results on the subdifferential of the
supremum. In this last case, the characterizations are given upon limit
processes on the $\varepsilon$-subdifferentials at the reference
point of the almost-active original functions. These limit processes also involve approximations by
finite-dimensional sections of the domain of the supremum function. 

The main results of this paper are applied to derive formulas for the subdifferential of the conjugate function (\cite{CH10}, \cite{CH12}, \cite{CH13}). Our approach permits simple proofs of these results, with the aim of relating the solution set of a nonconvex optimization problem and its convexified relaxation. Additionally, our results give rise to new Fritz-John and KKT conditions in convex semi-infinite programming.

The paper is organized as follows. After a short section introducing the
notation, in section \ref{preliminary} we present some
preliminary results in the continuous setting. In section \ref{seccom} we
apply our compactification approach to obtain, in Theorem
\ref{hamdkathir}, a first characterization of the subdifferential of the supremum.
Such a theorem constitutes an improved
version of the main result in \cite{CHL19c}, as the requirement of
equipping $T$ with a completely regular topology is eliminated. 
Theorem \ref{hamdkathir} is enhanced in Section \ref{sec5}, allowing for a more natural interpretation of the regularized functions.
The main result in section \ref{Secsubd} is Theorem
\ref{hamdkathirbisb}, involving only the $\varepsilon$-subdifferentials of the
original data functions. This theorem, whose proof is based on Lemmas
\ref{lemteco} and \ref{lemtec}, is crucial in the proposed approach to
move from the continuous to the non-continuous setting. Finally, in section \ref{Secap}, we give two applications.
The first one addresses the extension of the classical Fenchel duality
to nonconvex functions, and the second one establishes Fritz-John
and KKT optimality conditions for convex semi-infinite optimization.

\section{Notation}

Let\ $X$ be a (real) separated locally convex space, with topological dual $X^{\ast}$ endowed with the $w^{\ast}%
$-topology.$\ $By $\mathcal{N}_{X}$ ($\mathcal{N}_{X^{\ast}}$) we denote the
family of closed, convex, and balanced neighborhoods of the origin in $X$
($X^{\ast}$), also called $\theta$-neighborhoods. The spaces $X$ and $X^{\ast}$
are paired in duality by the bilinear form $(x^{\ast},x)\in X^{\ast}\times
X\mapsto\langle x^{\ast},x\rangle:=\left\langle x,x^{\ast}\right\rangle
:=x^{\ast}(x)$. The zero vectors in $X$ and $X^{\ast}$ are both denoted by
$\theta.$ We use the notation $\overline{\mathbb{R}}:=\mathbb{R}\cup
\{-\infty,+\infty\}$ and $\mathbb{R}_{\infty}:=\mathbb{R}\cup\{+\infty\}$, and
adopt the convention\emph{ }$\left(  +\infty\right)  +(-\infty)=\left(
-\infty\right)  +(+\infty)=+\infty.$

Given two nonempty sets $A$ and $B$ in $X$ (or in $X^{\ast}$), we define the
\emph{algebraic} (or \emph{Minkowski}) \emph{sum} by
\begin{equation}
A+B:=\{a+b:\ a\in A,\text{ }b\in B\},\quad A+\emptyset=\emptyset+A=\emptyset.
\label{mincov}%
\end{equation}

By $\operatorname*{co}(A),$ $\operatorname*{cone}(A),$ and
$\operatorname*{aff}(A)$, we denote the \emph{convex, }the\emph{ conical
convex, }and\emph{\ }the \emph{affine hulls} of the set $A$, respectively.
Moreover, $\operatorname*{int}(A)$ is the \emph{interior} of $A$, and
$\operatorname*{cl}A$ and $\overline{A}$ are indistinctly used for denoting
the \emph{closure }of $A$. We use $\operatorname*{ri}(A)$ to denote the
(topological) \emph{relative interior} of $A$ (i.e., the interior of $A$ in
the topology relative to $\operatorname*{aff}(A)$ if $\operatorname*{aff}(A)$ is
closed, and the empty set otherwise).

Associated with $A\neq\emptyset$ we consider the \emph{polar set} and the \emph{orthogonal subspace}
given respectively by
\[
A^{\circ}:=\left\{  x^{\ast}\in X^{\ast}:\ \langle x^{\ast},x\rangle\leq 1\text{
for all }x\in A\right\},
\]
and
\[
A^{\perp}:=\left\{  x^{\ast}\in X^{\ast}:\ \langle x^{\ast},x\rangle=0\text{
for all }x\in A\right\} .
\]
The following relation holds
\begin{equation}
\bigcap\nolimits_{L\in\mathcal{F}}(A+L^{\perp})\subset\operatorname*{cl}A,
\label{tr}%
\end{equation}
where $\mathcal{F}$ is the family\ of finite-dimensional linear subspaces in
$X.$

If $A\subset X$ is convex and $x\in X,$ we define the \emph{normal cone }to
$A$\ at $x$ as
\[
\mathrm{N}_{A}(x):=\left\{  x^{\ast}\in X^{\ast}:\ \langle x^{\ast}%
,z-x\rangle\leq0\text{ for all }z\in A\right\}  ,
\]
if\ $x\in A,$ and the empty set otherwise.

The basic concepts in this paper are traced from \cite{Mo65,Ro70}. Given a function $f:X\longrightarrow\overline{\mathbb{R}}$, its
\emph{(effective)} \emph{domain} and \emph{epigraph} are, respectively,
\[
\operatorname*{dom}f:=\{x\in X:\ f(x)<+\infty\}\text{ and }\operatorname*{epi}%
f:=\{(x,\lambda)\in X\times\mathbb{R}:\ f(x)\leq\lambda\}.
\]
We say that $f$ is \emph{proper} when\ $\operatorname*{dom}f\neq\emptyset$ and
$f(x)>-\infty$ for all $x\in X$. By $\operatorname*{cl}f$ and $\overline{\operatorname*{co}}f$ we respectively denote the \emph{closed} and the \emph{closed convex hulls} of $f$, which are the functions such that $\operatorname*{epi}%
(\operatorname*{cl}f)=\operatorname*{cl}(\operatorname*{epi}f)$ and $\operatorname*{epi}%
(\overline{\operatorname*{co}}f)=\overline{\operatorname*{co}}(\operatorname*{epi}f)$. We say that
$f$ is lower semicontinuous (lsc, for short) at $x$ if $(\operatorname*{cl}%
f)(x)=f(x),$ and lsc if $\operatorname*{cl}f=f.$

Given\ $x\in X$ and $\varepsilon\geq0,$ the $\varepsilon$%
-\emph{subdifferential }of $f$ at $x$ is
\[
\partial_{\varepsilon}f(x)=\{x^{\ast}\in X^{\ast}:\ f(y)\geq f(x)+\langle
x^{\ast},y-x\rangle-\varepsilon\text{ \ for all }y\in X\},
\]
when $x\in\operatorname*{dom}f,$ and $\partial_{\varepsilon}f(x):=\emptyset$
when $f(x)\notin\mathbb{R}.$ The\ elements of $\partial_{\varepsilon}f(x)$ are
called $\varepsilon$-\emph{subgradients} of $f$ at $x.$ The
\emph{subdifferential} of $f$ at $x$ is $\partial f(x):=\partial_{0}f(x),$
and its elements are called \emph{subgradients} of $f$ at $x$. If $f$ and $g$ are convex functions such that one of them is finite and continuous at a point
of the domain of the other one, then Moreau-Rockafellar's theorem says that
\begin{equation}
\partial(f+g)=\partial f+\partial g. \label{MR}%
\end{equation}
Given a function $f:X\to\overline{\mathbb{R}}$, the \emph{(Fenchel) conjugate} of $f$ is the function $f^*:X^*\to\overline{\mathbb{R}}$ defined as 
$$
f^*(x^*):=\sup_{x\in X}\{\langle x^*,x\rangle-f(x)\}.
$$
The \emph{indicator} and the \emph{support} functions of $A\subset X$ are respectively 
defined as
\[
\mathrm{I}_{A}(x):=\left\{
\begin{array}
[c]{ll}%
0 ,& \text{if }x\in A,\\
+\infty , & \text{if }x\in X\setminus A,
\end{array}
\right.
\]
and 
\[
\sigma_{A}:=\mathrm{I}_{A}^*.
\]
Provided that $f^{\ast}$ is proper, by Moreau's theorem we have
\begin{equation}
f^{\ast\ast}=\overline{\operatorname*{co}}f, \label{moreau}%
\end{equation}
where $f^{\ast\ast}:=(f^{\ast})^*$.
For example, if $\left\{  f_{i},\text{ }i\in I\right\}  $ is a nonempty family
of proper lsc convex functions, then
\begin{equation}
(\sup\nolimits_{i\in I}f_{i})^{\ast}=\overline{\operatorname*{co}}(\inf\nolimits_{i\in I}%
f_{i}^{\ast}), \label{moreaub}%
\end{equation}
provided that the supremum function $\sup_{i\in I}f_{i}$ is proper. Thus,
given a nonempty family of\ closed convex sets $A_{i}\subset X,$
\thinspace$i\in I,$ such that $\cap_{i\in I}A_{i}\neq\emptyset,$ we have
$\mathrm{I}_{\cap_{i\in I}A_{i}}(x)=\sup_{i\in I}\mathrm{I}_{A_{i}}(x)$ and,
so, by taking the conjugate in the equalities $\mathrm{I}_{\cap_{i\in I}A_{i}%
}(x)=\sup_{i\in I}\mathrm{I}_{A_{i}}(x)=\sup_{i\in I}\mathrm{\sigma}_{A_{i}%
}^{\ast}(x),$ we obtain
\[
\mathrm{\sigma}_{\cap_{i\in I}A_{i}}=(\mathrm{I}_{\cap_{i\in I}A_{i}})^{\ast
}=(\sup\nolimits_{i\in I}\mathrm{I}_{A_{i}})^{\ast}=\overline{\operatorname*{co}}%
(\inf\nolimits_{i\in I}\mathrm{\sigma}_{A_{i}}).
\]

\section{Preliminary results in the continuous framework\label{preliminary}}

In Section \ref{seccom} we develop a compactification process
addressed to give\ new characterizations of the subdifferential of
the pointwise supremum, with the aim of unifying both the compact and
non-compact settings. In this section we gather some preliminary results in
the continuous setting.

Given the family of convex functions $f_{t}:X\rightarrow
\overline{\mathbb{R}},$ \ $t\in T,$ and the supremum function $f:=\sup_{t\in
T}f_{t}$, we start from the following characterization of $\partial
f(x)$ in the continuous setting,\ given in \cite[Proposition 2]{CHL19},
where the following notation is used:
\begin{equation}
\mathcal{F}(x):=\left\{  L\subset X:\text{ }L\text{ is a finite-dimensional
linear subspace containing }x\right\}  , \label{finitesub}%
\end{equation}
and
\[
T_{\varepsilon}(x):=\left\{  t\in T: f_{t}(x)\geq f(x)-\varepsilon\right\}
,
\]
for $\varepsilon\geq0;$ we set $T(x):=T_{0}(x).$

\begin{prop}
\label{thmcompact0}\emph{\cite[Proposition 2]{CHL19}} Fix\ $x\in X$ and
$\varepsilon>0$ such that $T_{\varepsilon}(x)$ is compact Hausdorff and,\ for
each net $(t_{i})_{i}\subset T_{\varepsilon_{0}}(x)$ converging to $t,$
\begin{equation}
\limsup\nolimits_{i}f_{t_{i}}(z)\leq f_{t}(z)\text{ \ \ for all }%
z\in\operatorname*{dom}f; \label{uscm}%
\end{equation}
that is, the functions $f_{(\cdot)}(z)$ are upper semi-continuos (usc, in brief) relatively to $T_{\varepsilon
_{0}}(x).$ Then we have%
\begin{equation}
\partial f(x)=\bigcap\nolimits_{L\in\mathcal{F}(x)}\operatorname*{co}\left\{
\bigcup\nolimits_{t\in T(x)}\partial(f_{t}+\mathrm{I}_{L\cap
\operatorname*{dom}f})(x)\right\}  . \label{form5}%
\end{equation}

\end{prop}

It is worth recalling that the intersection over the $L$'s in (\ref{form5}) is
removed in finite dimensions (\cite[Theorem 3]{CHL19}) and, more generally, if
$\operatorname*{ri}(\operatorname*{dom}f)\neq\emptyset$ and $f_{\mid
\operatorname*{aff}(\operatorname*{dom}f)}$ is continuous on
$\operatorname*{ri}(\operatorname*{dom}f),$ then we have (see \cite[Corollary
3.9]{CHL19b})
\[
\partial f(x)=\overline{\operatorname*{co}}\left\{  \bigcup\nolimits_{t\in
T(x)}\partial(f_{t}+\mathrm{I}_{\operatorname*{dom}f})(x)\right\}  .
\]
Consequently, if $f$ is continuous somewhere in its domain, then
(\cite[Theorem 3.12]{CHL19b})
\[
\partial f(x)=\overline{\operatorname*{co}}\left\{  \bigcup\nolimits_{t\in
T(x)}\partial f_{t}(x)\right\}  +\mathrm{N}_{\operatorname*{dom}f}(x),
\]
and the closure is removed in finite dimensions. In particular, when $f$ is
continuous at the reference point $x,$ the normal cone above collapses to $\theta$ and we
recover Valadier's formula in \cite{Va69}.

On the other hand, in the general setting, when either $T$ is not compact
and/or some of the mappings $t\rightarrow f_{t}(z),$ $z\in\operatorname*{dom}%
f,$ fail to be usc, the active index set $T(x)$ as well as the subdifferential
sets $\partial f_{t}(x)$ may be empty. To overcome\ this situation, the following result given in
\cite[Theorem 4]{HLZ08} (see, also, \cite{HL08} for finite dimensions) appeals
to\ the $\varepsilon$-active set $T_{\varepsilon}(x)$ and the $\varepsilon$-subdifferentials.

\begin{prop}
\label{p1}If
\begin{equation}\label{cll}
\operatorname*{cl}f=\sup_{t\in T}(\operatorname*{cl}f_{t}), %
\end{equation}
then for every\ $x\in X$%
\begin{equation}
\partial f(x)=\bigcap\nolimits_{\varepsilon>0,L\in\mathcal{F}(x)}%
\overline{\operatorname*{co}}\left\{  \bigcup\nolimits_{t\in T_{\varepsilon
}(x)}\partial_{\varepsilon}f_{t}(x)+\mathrm{N}_{L\cap\operatorname*{dom}%
f}(x)\right\}  . \label{fe1}%
\end{equation}

\end{prop}

Also here, the intersection over the $L$'s is dropped out if
$\operatorname*{ri}(\operatorname*{dom}f)\neq\emptyset$ (\cite[Corollary
8]{HLZ08}). Moreover, if $f$ is continuous somewhere, so that (\ref{cll}) holds
automatically (\cite[Corollary 9]{HLZ08}), then the last formula reduces to
\[
\partial f(x)=\mathrm{N}_{\operatorname*{dom}f}(x)+ \bigcap\nolimits_{\varepsilon>0}\overline
{\operatorname*{co}}\left\{  \bigcup\nolimits_{t\in T_{\varepsilon}%
(x)}\partial_{\varepsilon}f_{t}(x)\right\}  .
\]
Hence, provided that\ $f$ is continuous at $x,$ we obtain the formula in \cite{Vo94} (where the underlying space $X$ is additionally assumed to be normed).

Condition (\ref{cll}) guarantees  the possibility of characterizing $\partial f(x)$ by means
of the $f_{t}$'s, and not via the augmented functions $f_{t}+\mathrm{I}_{L\cap
\operatorname*{dom}f}$ as in Proposition \ref{thmcompact0}. Thus, to complete the
analysis, we give next a consequence of (\ref{fe1}), which avoids to appeal to condition
(\ref{cll}).

\begin{prop}
\label{p2}For every\ $x\in X,$%
\begin{equation}
\partial f(x)=\bigcap\nolimits_{\varepsilon>0,L\in\mathcal{F}(x)}%
\overline{\operatorname*{co}}\left\{  \bigcup\nolimits_{t\in T_{\varepsilon
}(x)}\partial_{\varepsilon}(f_{t}+\mathrm{I}_{L\cap\operatorname*{dom}%
f})(x)\right\}  . \label{fe2}%
\end{equation}

\end{prop}

\begin{proof}
Fix $x\in\operatorname*{dom}f$ and $L\in\mathcal{F}(x),$ and denote
\[
g_{t}:=f_{t}+\mathrm{I}_{L\cap\operatorname*{dom}f},\text{ }t\in T;\quad
g:=\sup\nolimits_{t\in T}g_{t}.
\]
We have\ $\operatorname*{dom}g_{t}=L\cap\operatorname*{dom}f$ and%
\[
\operatorname*{dom}g\cap(\cap_{t\in T}\operatorname*{ri}(\operatorname*{dom}%
g_{t}))=(L\cap\operatorname*{dom}f)\cap\operatorname*{ri}(\operatorname*{dom}%
f\cap L)=\operatorname*{ri}(\operatorname*{dom}f\cap L)\neq\emptyset,
\]
so that, by \cite[Corollary 9(iv)]{HLZ08}, the family $\left\{  g_{t},\text{
}t\in T\right\}  $ satisfies condition (\ref{cll}). At the same time we have,
for all $\varepsilon\geq0,$
\[
\left\{  t\in T: g_{t}(x)\geq g(x)-\varepsilon\right\}  =T_{\varepsilon
}(x).
\]
Then, since that 
$\partial f(x)\subset\partial(f+\mathrm{I}_{L\cap\operatorname*{dom}%
f})(x)=\partial g(x)$,
by Proposition \ref{p1} we obtain that 
\begin{align*}
\partial f(x)  &  \subset\bigcap\nolimits_{\varepsilon>0}%
\overline{\operatorname*{co}}\left\{  \bigcup\nolimits_{t\in T_{\varepsilon
}(x)}\partial_{\varepsilon}g_{t}(x)+\mathrm{N}_{L\cap\operatorname*{dom}%
g}(x)\right\} \\
&  \subset\bigcap\nolimits_{\varepsilon>0}\overline{\operatorname*{co}%
}\left\{  \bigcup\nolimits_{t\in T_{\varepsilon}(x)}\partial_{\varepsilon
}(f_{t}+\mathrm{I}_{L\cap\operatorname*{dom}f})(x)\right\}  ,
\end{align*}
and the inclusion \textquotedblleft$\subset$\textquotedblright\ in (\ref{fe2})
follows as $L$ was arbitrarily chosen.\ The opposite inclusion\ is
straightforward, and we are done.
\end{proof}

\section{Compactification approach to the subdifferential\label{seccom}}

Our main objective in this section is to give a new characterization for
$\partial f(x),$ which covers both formula (\ref{form5}) in the
compact-continuous setting, using the active set and the exact
subdifferential, and formula (\ref{fe2}) in the non-compact non-continuous
framework, given in terms\ of $\varepsilon$-active indices and $\varepsilon
$-subdifferentials. To this aim, we develop a compactification approach which works by
extending the original index set $T$ to a compact set $\widehat{T}$, and building new
appropriate functions $f_{\gamma}$, $\gamma\in \widehat{T}$, that satisfy property
(\ref{uscm}) of Proposition \ref{thmcompact0}. To make the paper
self-contained, we resume here the main features of the compactification
process, which can be also found in \cite{CHL19c}.

We start by assuming that $T$ is endowed with some topology $\tau,$ for
instance the discrete topology. If
\begin{equation}
\mathcal{C}(T,\left[  0,1\right]  ):=\left\{  \varphi:T\rightarrow\left[
0,1\right]  :\varphi\text{ is }\tau\text{-continuous}\right\}  ,
\label{ccc}%
\end{equation}
we consider the product space $\left[  0,1\right]  ^{C(T,\left[  0,1\right]
)},$ which is compact for the product topology (by Tychonoff\ theorem).
We\ regard\ the index set $T$ as a subset of $\left[  0,1\right]
^{C(T,\left[  0,1\right]  )}$. For this purpose we consider the continuous embedding
$\mathfrak{w}:T\rightarrow\left[  0,1\right]  ^{C(T,\left[  0,1\right]  )}$
which assigns to each $t\in T$ the evaluation function $\mathfrak{w}%
(t)=\gamma_{t},$
defined as
\begin{equation}
\gamma_{t}(\varphi):=\varphi(t),\text{ \ }\varphi\in\mathcal{C}(T,\left[
0,1\right]  ). \label{eval}%
\end{equation}
The closure of $\mathfrak{w(}T)$ in $\left[  0,1\right]  ^{C(T,\left[
0,1\right]  )}$ for the product topology is the compact set
\begin{equation}
\widehat{T}:=\operatorname*{cl}(\mathfrak{m}(T)), \label{ttilda}%
\end{equation}
which is\ the so-called \emph{Stone-\v{C}ech compactification} of $T,$ also denoted by $\beta T.$ The convergence in $\widehat{T}$ is the
pointwise convergence; i.e., for $\gamma\in\widehat{T}$ and a net
$(\gamma_{i})_{i}\subset\widehat{T}$ we have $\gamma_{i}\rightarrow\gamma$ if and only if
\begin{equation}
\gamma_{i}(\varphi)\rightarrow\gamma(\varphi)\text{ \ for all }\varphi
\in\mathcal{C}(T,\left[  0,1\right]  ). \label{conve}%
\end{equation}
Hence, provided that $T$ is completely regular (when endowed with the discrete topology, for isntance), the mapping $\mathfrak{w}$ is an homeomorphism between $T$ and $\mathfrak{w}(T)$, and if $\gamma_{i}=\gamma_{t_{i}}$ and $\gamma=\gamma_{t}$ for
some $t,t_{i}\in T,$ then $\gamma_{i}\rightarrow\gamma$ if and only if
$t_{i}\rightarrow t$ in $T.$

Next, we\ enlarge the original family $\left\{  f_{t},\text{ }t\in T\right\}
$ by introducing\ the functions $f_{\gamma}:X\rightarrow\overline{\mathbb{R}%
},$ $\gamma\in\widehat{T},$ defined\ by
\begin{equation}
f_{\gamma}(z):=\limsup\nolimits_{\gamma_{t}\rightarrow\gamma,\text{ }t\in
T}f_{t}(z). \label{fgam}%
\end{equation}
It can be easily verified that the\ functions $f_{\gamma},$ $\gamma
\in\widehat{T},$ are all convex and satisfy $\sup\nolimits_{\gamma
\in\widehat{T}}f_{\gamma}\leq f.$ Moreover, if\ $(t_{n})_{n}\subset T$
verifies\ $f(z)=\lim_{n}f_{t_{n}}(z),$ with $z\in X,$ then there exist a
subnet $(t_{i})_{i}$ of $(t_{n})_{n}$ and $\gamma\in\widehat{T}$ such that
$\gamma_{t_{i}}\rightarrow\gamma.$ Hence,
\[
f_{\gamma}(z)\geq\limsup\nolimits_{i}f_{t_{i}}(z)=\lim\nolimits_{i}f_{t_{i}%
}(z)=\lim\nolimits_{n}f_{t_{n}}(z)=f(z),
\]
and so $\sup_{\gamma\in\widehat{T}}f_{\gamma}\geq f.$ In other words,
the\ functions $f_{\gamma}$ provide the same supremum $f$ as the original
$f_{t}$'s,%
\[
\sup\nolimits_{\gamma\in\widehat{T}}f_{\gamma}=\sup\nolimits_{t\in T}f_{t}=f.
\]
If $f(x)\in\mathbb{R}$ and $\varepsilon\geq0,$ then the extended
$\varepsilon$-active index set of $f$ at $x$ is
\begin{equation}
\widehat{T}_{\varepsilon}(x):=\left\{  \gamma\in\widehat{T}:f_{\gamma}(x)\geq
f(x)-\varepsilon\right\}  , \label{dgam}%
\end{equation}
with $\widehat{T}(x):=\widehat{T}_{0}(x);$ when $f(x)\not \in \mathbb{R}$ we
set\ $\widehat{T}_{\varepsilon}(x):=\emptyset$ for all $\varepsilon\geq0.$
By\ the compactness of $\widehat{T}$ and the simple fact that, for each $t\in
T$,%
\begin{eqnarray*}
f_{\gamma_{t}}(x)=\limsup_{\gamma_{s}\rightarrow\gamma_{t}}f_{s}%
(x)&= &\sup\left\{  \lim_{i}f_{t_{i}}(x),\text{ }\gamma_{t_{i}}\rightarrow
\gamma_{t}\right\} \\
&\geq&\sup\left\{  \lim_{i}f_{t_{i}}(x),\text{ }t_{i}\rightarrow t\right\}  \geq f_{t}(x),
\end{eqnarray*}
we verify that\ $\widehat{T}_{\varepsilon}(x)\neq\emptyset.$ Also, the
closedness of $\widehat{T}_{\varepsilon}(x)$ comes by using a diagonal process.

The way that the functions $f_{\gamma},$ $\gamma\in\widehat{T}$, are
constructed ensures the fulfillment of the upper semi-continuity property
required\ in Proposition \ref{thmcompact0}. More precisely, assuming\ that
$f(x)\in\mathbb{R}$ and $\varepsilon\geq0,$ for every net $(\gamma_{i}%
)_{i}\subset\widehat{T}_{\varepsilon}(x)$ with an accumulation point $\gamma
\in\widehat{T}_{\varepsilon}(x)$, and every $z\in\operatorname*{dom}f$, we
verify that
\begin{equation}
\limsup\nolimits_{i}f_{\gamma_{i}}(z)\leq f_{\gamma}(z). \label{usc}%
\end{equation}
Indeed, we may assume without loss of generality\ that $\gamma_{i}\rightarrow\gamma$ and $\limsup\nolimits_{i}%
f_{\gamma_{i}}(z)=\lim_{i}f_{\gamma_{i}}(z)=\alpha\in\mathbb{R}$. Next, for each $i$ there exists a net
$(t_{ij})_{j}\subset T$ such that
\[
\gamma_{t_{ij}}\rightarrow_{j}\gamma_{i},\text{ \ }f_{\gamma_{i}}%
(z)=\lim\nolimits_{j}f_{t_{ij}}(z);
\]
that is, $(\gamma_{t_{ij}},f_{t_{ij}}(z))\rightarrow_{j}(\gamma_{i}%
,f_{\gamma_{i}}(z))$ and $(\gamma_{i},f_{\gamma_{i}}(z))\rightarrow_{i}%
(\gamma,\alpha).$ Then we can find a diagonal net $(t_{ij_{i}})_{i}\subset T$
such that $(\gamma_{t_{ij_{i}}},f_{t_{ij_{i}}}(z))\rightarrow_{i}%
(\gamma,\alpha),$ and we obtain
\[
f_{\gamma}(z)\geq\limsup\nolimits_{i}f_{t_{ij_{i}}}(z)=\alpha=\limsup
\nolimits_{i}f_{\gamma_{i}}(z).
\]

The compactification process above covers in a natural way the compact
framework.\ Namely, if $T$ is compact Hausdorff (hence, complete regular), then
the family $\left\{  f_{\gamma},\gamma\in\widehat{T}\right\}  $ above turns
out to be the family of the usc regularization of the functions $f_{(\cdot)}(z)$, given by
\[
\bar{f}_{t}(z):=\limsup_{s\rightarrow t}f_{s}(z).
\]
In this case, the indexed set $T$ does not change; i.e., $\widehat{T}=T.$
Consequently, if additionally the functions $f_{(\cdot)}(z),$ $z\in
\operatorname*{dom}f,$ are already usc,$\ $then we recover the classical
compact and continuous setting, originally\ proposed in \cite{Va69}.

The following theorem characterizes\ $\partial f(x)$ in terms of the functions
$f_{\gamma}$ (see (\ref{fgam})) and the compact set $\widehat{T}(x),$ when
$\tau$ is any topology on $T$. This result is crucial in the subsequent sections.

\begin{theo}
\label{hamdkathir}Let $f_{t}:X\rightarrow\overline{\mathbb{R}},$ $t\in T,$
be\ convex functions and $f=\sup_{t\in T}f_{t}.$ Then, for every $x\in X,$%
\begin{equation}
\partial f(x)=\bigcap\nolimits_{L\in\mathcal{F}(x)}\operatorname*{co}\left\{
\bigcup\nolimits_{\gamma\in\widehat{T}(x)}\partial(f_{\gamma}+\mathrm{I}%
_{L\cap\operatorname*{dom}f})(x)\right\}  . \label{mainformulaalternative}%
\end{equation}

\end{theo}

\begin{proof}
First, we consider\ that\ the topology $\tau$ in $T$ is the discrete topology
$\tau_{d}$, so that\ $C(T,\left[  0,1\right]  ):=\left[  0,1\right]  ^{T}$ and
$\widehat{T}$ is compact.\ Moreover, since $(T,\tau_{d})$ is completely
regular, $\widehat{T}$ is Hausdorff (see, i.e., \cite[$\cal{x}$38]{Mu00}). Since
$f=\sup_{\gamma\in\widehat{T}}f_{\gamma}$ and (\ref{usc}) holds, Proposition
\ref{thmcompact0} applies and yields%
\begin{equation}
\partial f(x)=\bigcap\nolimits_{L\in\mathcal{F}(x)}\operatorname*{co}\left\{
\bigcup\nolimits_{\gamma\in\widehat{T}^{d}(x)}\partial(f_{\gamma}%
^{d}+\mathrm{I}_{L\cap\operatorname*{dom}f})(x)\right\}  , \label{et}%
\end{equation}
where\ $f_{\gamma}^{d}$ and $\widehat{T}^{d}(x)$ are defined as in
(\ref{fgam}) and (\ref{dgam}), respectively,\ but with respect to the topology
$\tau_{d}.$

Now, let $\tau$ be any topology, so that $\tau\subset\tau_{d}$ and, for any
$(\gamma_{t_{i}})_{i}\subset\widehat{T},$%
\begin{align*}
\gamma_{t_{i}}\rightarrow_{\tau_{d}}\gamma &  \Longleftrightarrow\varphi
(t_{i})\rightarrow\gamma(\varphi)\text{ for all }\varphi\in\left[  0,1\right]
^{T}\\
&  \implies\varphi(t_{i})\rightarrow\gamma(\varphi)\text{ for all }\varphi\in
C(T,\left[  0,1\right]  )\\
&  \Longleftrightarrow\gamma_{t_{i}}\rightarrow_{\tau}\gamma;
\end{align*}
hence, for every $z\in X,$%
\[
f_{\gamma}^{d}(z)=\limsup_{\gamma_{t}\rightarrow_{\tau_{d}}\gamma,\text{ }t\in
T}f_{t}(z)\leq\limsup_{\gamma_{t}\rightarrow_{\tau}\gamma,\text{ }t\in T}%
f_{t}(z)=f_{\gamma}(z).
\]
Moreover, since for all $\gamma\in\widehat{T}^{d}(x)$ we have
\[
f(x)=f_{\gamma}^{d}(x)\leq f_{\gamma}(x)\leq f(x),
\]
we deduce that
\begin{equation}
\widehat{T}^{d}(x)\subset\widehat{T}(x)\text{ and }\partial(f_{\gamma}%
^{d}+\mathrm{I}_{L\cap\operatorname*{dom}f})(x)\subset\partial(f_{\gamma
}+\mathrm{I}_{L\cap\operatorname*{dom}f})(x). \label{ad}%
\end{equation}
Thus, by (\ref{et}),%
\[
\partial f(x)\subset\bigcap\nolimits_{L\in\mathcal{F}(x)}\operatorname*{co}%
\left\{  \bigcup\nolimits_{\gamma\in\widehat{T}(x)}\partial(f_{\gamma
}+\mathrm{I}_{L\cap\operatorname*{dom}f})(x)\right\}  ,
\]
and (\ref{mainformulaalternative}) follows as the opposite inclusion is straightforward.
\end{proof}

It is worth observing, from the inclusions in\ (\ref{ad}), that the discrete
topology provides the simplest characterization of $\partial f(x),$ since it
possibly involves less and smaller sets. Also observe that the intersection
over finite-dimensional $L$ in (\ref{mainformulaalternative}) is superfluous
in finite dimensions.

Theorem \ref{hamdkathir} covers the classical Valadier's setting where $T$ is
compact Hausdorff and the mappings $f_{(\cdot)}(z),$ $z\in\operatorname*{dom}%
f,$ are usc. In this case, formula (\ref{mainformulaalternative}) reduces to
(see Proposition \ref{thmcompact0})%
\[
\partial f(x)=\bigcap\nolimits_{L\in\mathcal{F}(x)}\operatorname*{co}\left\{
\bigcup\nolimits_{\gamma\in T(x)}\partial(f_{t}+\mathrm{I}_{L\cap
\operatorname*{dom}f})(x)\right\}  .
\]

Let us also observe that when $T$ admits a one-point compactification
$T_{{\small \Omega}}:=T\cup\left\{  {\small \Omega}\right\}  $
$({\small \Omega}\notin T)$, which occurs if and only if $T$ is locally
compact Hausdorff (hence, complete regular), instead of $\left\{  f_{\gamma},\text{ }\gamma\in\widehat{T}\right\}
$ we can use the family $\left\{  f_{\gamma_{t}},\text{ }t\in T;\text{
}f_{{\small \Omega}}\right\}  ,$ where
\begin{equation}
f_{{\small \Omega}}(z):=\limsup_{t\rightarrow{\small \Omega}}f_{t}(z),\text{
}z\in X. \label{fomega}%
\end{equation}
Indeed, in this case the Stone-\v{C}ech compactification of $T$ is
\[
\widehat{T}:=\left\{  \gamma_{t},\text{ }t\in T\right\}  \cup\left\{  \lim
_{i}\gamma_{t_{i}}:(t_{i})_{i}\subset T,\text{ }t_{i}\rightarrow
{\small \Omega}\right\}  ,
\]
where the limits $\lim_{i}\gamma_{t_{i}}$ and $t_{i}\rightarrow{\small \Omega
}$ are in $\left[  0,1\right]  ^{C(T,\left[  0,1\right]  )}$ and
$T_{{\small \Omega}},$ respectively. In this way we obtain, for all $t\in T,$
\begin{equation}
f_{\gamma_{t}}=\limsup_{\gamma_{s}\rightarrow\gamma_{t},\text{ }s\in T}%
f_{s}=\limsup_{s\rightarrow t,\text{ }s\in T}f_{s},\text{ for }t\in T,
\label{fomegat}%
\end{equation}
due to the topological identification of $T$ with $\mathfrak{w}(T)$, and
\[
f_{\gamma}=\limsup_{\gamma_{t}\rightarrow\gamma,\text{ }t\in T}f_{t}%
=\limsup_{\gamma_{t}\rightarrow\gamma,\text{ }t\rightarrow{\small \Omega
},\text{ }t\in T}f_{t}\text{, for }\gamma\in\widehat{T}\setminus T.\text{ }%
\]
Now, we observe that
\[
\sup_{\gamma\in\widehat{T}\setminus T}f_{\gamma}=\sup_{\gamma\in
\widehat{T}\setminus T}\limsup_{\gamma_{t}\rightarrow\gamma,\text{
}t\rightarrow{\small \Omega},\text{ }t\in T}f_{t}=\limsup_{t\rightarrow
{\small \Omega}}f_{t}=f_{{\small \Omega}}.
\]
It is clear that the family $\left\{  f_{\gamma_{t}},\text{ }t\in T;\text{
}f_{{\small \Omega}}\right\}  $ and the (one-point compactification) index set $T\cup\{\Omega\}$ satisfy the assumption of Proposition
\ref{thmcompact0}, together with $f=\sup\left\{  f_{\gamma_{t}},\text{ }t\in
T;\text{ }f_{{\small \Omega}}\right\}  $. Thus, it suffices to consider Theorem \ref{hamdkathir}
with this new family $\left\{  f_{\gamma_{t}},\text{ }t\in T;\text{
}f_{{\small \Omega}}\right\}  $ instead of the one of the original $f_{\gamma}$'s.

In the particular case when $T=\mathbb{N}$, endowed with\ the
discrete topology, for each $n\in\mathbb{N}$ we obtain
\[
f_{\gamma_{n}}=\limsup_{\gamma_{k}\rightarrow\gamma_{n},\text{ }k\in
\mathbb{N}}f_{k}=\limsup_{k\rightarrow n,\text{ }k\in\mathbb{N}}f_{k}=f_{n},
\]
so that the family to consider in Theorem \ref{hamdkathir} is
\[
\left\{  f_{n},\text{ }n\in\mathbb{N};\text{ }f_{\infty}\right\}  ,
\]
where
\[
f_{{\small \infty}}=\limsup_{n\rightarrow\infty}f_{n}.
\]

\begin{cor}
Assume that $T$ is locally compact Hausdorff. Then for every $x\in X$ formula
(\ref{mainformulaalternative}) holds with
\[
\widehat{T}(x)=\left\{
\begin{array}
[c]{ll}%
\left\{  \gamma_{t},\text{ }t\in T,\text{ }f_{\gamma_{t}}(x)=f(x)\right\}  , &
\text{if }f_{{\small \Omega}}(x)<f(x),\\
\left\{  \gamma_{t},\text{ }t\in T,\text{ }f_{\gamma_{t}}(x)=f(x),\text{
}{\small \Omega}\right\}  , & \text{if }f_{{\small \Omega}}(x)=f(x),
\end{array}
\right.
\]
and, when $T=\mathbb{N}$,%
\[
\widehat{T}(x)=\left\{
\begin{array}
[c]{ll}%
\left\{  n\in\mathbb{N},\text{ }f_{n}(x)=f(x)\right\}  , & \text{if }%
f_{\infty}(x)<f(x),\\
\left\{  n\in\mathbb{N},\text{ }f_{n}(x)=f(x),\text{ }{\small \infty}\right\}
, & \text{if }f_{{\small \infty}}(x)=f(x).
\end{array}
\right.
\]

\end{cor}

\section{From non-continuous to continuous. Enhanced formulas}\label{sec5}

We give in this section some new\ characterizations of $\partial f(x),$ which
provide additional insight\ to Theorem \ref{hamdkathir} and that are
applied\ in Section \ref{Secsubd}.

According to Theorem \ref{hamdkathir}, $\partial f(x)$ only involves the
active functions $f_{\gamma},$ i.e., when\ $\gamma\in\widehat{T}(x).$ The idea
behind the following result is to replace these $f_{\gamma}$'s by the\ new functions
$\tilde{f}_{\gamma}:X\rightarrow\mathbb{R}_{\infty},$ $\gamma\in\widehat{T},$
defined as
\begin{equation}
\tilde{f}_{\gamma}(z):=\limsup\limits_{\gamma_{t}\rightarrow\gamma
,f_{t}(x)\rightarrow f(x),\text{ }t\in T}f_{t}(z), \label{fgamat}%
\end{equation}
considering only those\ nets $(t_{i})_{i}\subset T$ associated with\ functions
$f_{t_{i}}$ approaching the supremum function $f$ at the nominal point $x.$
Observe that if $\gamma\in\widehat{T}\setminus\widehat{T}(x),$ then $\tilde
{f}_{\gamma}\equiv-\infty$ by the convention $\sup\emptyset=-\infty,$ and this
function is ignored when taking the supremum. 

Remember that $T$ is endowed with any topology.
\begin{theo}\label{thm6}
For every $x\in X$ we have
\begin{equation}
\partial f(x)=\bigcap\nolimits_{L\in\mathcal{F}(x)}\operatorname*{co}\left\{
\bigcup\nolimits_{\gamma\in\widehat{T}(x)}\partial(\tilde{f}_{\gamma
}+\mathrm{I}_{L\cap\operatorname*{dom}f})(x)\right\}  ,
\label{mainformulaalternativeb}%
\end{equation}
where $\tilde{f}_{\gamma}$ and $\widehat{T}(x)$ are defined in (\ref{fgamat})
and (\ref{dgam}), respectively.
\end{theo}

\begin{proof}
We only need to check the inclusion \textquotedblleft$\subset$
\textquotedblright\ when $\tau$ is the discrete topology $\tau_d$, and $\partial f(x)\neq\emptyset;$ hence, $f$ is lsc at
$x$ and proper, and we may suppose, without loss of generality, that $x=\theta$ and \ $f(\theta
)=0.$ Let us\ fix a closed convex neighborhood $U$ of $\theta$ such that
$f(z)\geq-1$, for all $z\in U,$ and denote by $g_{t}:X\rightarrow
\mathbb{R}_{\infty},$ $t\in T,$ the functions given by
\begin{equation}
g_{t}(z):=\max\left\{  f_{t}(z),-1\right\}  . \label{gt}%
\end{equation}
Thus, for all $z\in U,$
\[
f(z)=\max\left\{  f(z),-1\right\}  =\sup\nolimits_{t\in T}\max\left\{
f_{t}(z),-1\right\}  =\sup\nolimits_{t\in T}g_{t}(z),
\]
and so, applying (\ref{mainformulaalternative}) with the discrete topology
$\tau_{d\text{ }}$on $T$ to the family $\left\{  g_{t},\text{ }t\in T\right\}
,$
\begin{equation}
\partial f(\theta)=\partial(\sup\nolimits_{t\in T}g_{t})(\theta)=\bigcap
\nolimits_{L\in\mathcal{F}(\theta)}\operatorname*{co}\left\{  \bigcup
\nolimits_{\gamma\in\widetilde{T}(\theta)}\partial(g_{\gamma}+\mathrm{I}%
_{L\cap\operatorname*{dom}f})(\theta)\right\}  , \label{im}%
\end{equation}
where $g_{\gamma}:=\limsup\limits_{\gamma_{t}\rightarrow\gamma,\text{ }t\in
T}g_{t}$ and $\widetilde{T}(\theta):=\left\{  \gamma\in\widehat{T}:
g_{\gamma}(\theta)=0\right\}  .$

Let us first\ verify that
\begin{equation}
\widetilde{T}(\theta)=\widehat{T}(\theta). \label{tta}%
\end{equation}
Indeed, if $\gamma\in\widetilde{T}(\theta)$ so that
\[
0=g_{\gamma}(\theta)=\limsup\limits_{\gamma_{t}\rightarrow\gamma,\text{ }t\in
T}g_{t}(\theta)\leq\max\left\{  f_{\gamma}(\theta),-1\right\}  \leq
\max\left\{  f(\theta),-1\right\}  =0,
\]
then $f_{\gamma}(\theta)=0$ and, so, $\gamma\in\widehat{T}(\theta).$
Conversely, if $\gamma\in\widehat{T}(\theta),$ then
\[
0=f_{\gamma}(\theta)\leq g_{\gamma}(\theta)\leq\sup_{\gamma\in\widehat{T}%
}g_{\gamma}(\theta)=\sup_{t\in T}g_{t}(\theta)=f(\theta)=0,
\]
and so $\gamma\in\widetilde{T}(\theta).$

Next, we fix\ $\gamma\in\widetilde{T}(\theta)$ and, by the definition of this set, let
$(\bar{t}_{i})_{i}\subset T$ be a net such that $\gamma_{\bar{t}_{i}%
}\rightarrow\gamma$\ and $\lim_{i}g_{\bar{t}_{i}}(\theta)=0;$ hence,
\begin{equation}
\lim_{i}f_{\bar{t}_{i}}(\theta)=\lim_{i}g_{\bar{t}_{i}}(\theta)=0. \label{nt}%
\end{equation}
We\ also introduce the functions $\varphi_{z},$ $z\in\operatorname*{dom}f,$
defined on $T$ as follows\
\[
\varphi_{z}(t):=(\max\left\{  f(z)+1,1\right\}  )^{-1}(g_{t}(z)+1),
\]
which are ($\tau_{d}$-)continuous functions such that $\varphi_{z}%
(t)\in\left[  0,1\right]  $ for all $t\in T,$ because
\[
-1\leq g_{t}(z)\leq\max\left\{  f(z),-1\right\}  <+\infty\text{\ for all }t\in
T\text{ and }z\in\operatorname*{dom}f.
\]
Hence, for every $\gamma_{t_{i}}\rightarrow\gamma$ we have $\varphi_{z}%
(t_{i})\rightarrow_{i}\gamma(\varphi_{z}),$ and this entails\
\begin{equation}
g_{t_{i}}(z)\rightarrow_{i}-1+(\max\left\{  f(z)+1,1\right\}  )\gamma
(\varphi_{z})\in\mathbb{R}. \label{ns}%
\end{equation}
Consequently, by taking into account that $\gamma_{\bar
{t}_{i}}\rightarrow\gamma$ and $\lim_{i}f_{\bar{t}_{i}}(\theta)=0$ (see
(\ref{nt})) we obtain
\begin{equation}\label{be1}
g_{\gamma}=\limsup\limits_{\gamma
_{t}\rightarrow\gamma,\text{ }t\in T}g_{t}=\lim_{\gamma_{t}\rightarrow\gamma}g_{t}
=\lim_{\gamma_{t}\rightarrow\gamma,\text{
}f_{t}(\theta)\rightarrow0}g_{t},
\end{equation}
which leads us to
\begin{eqnarray}
g_{\gamma}+\mathrm{I}_{L\cap\operatorname*{dom}%
f}&=&\lim\limits_{\gamma_{t}\rightarrow\gamma,f_{t}(\theta)\rightarrow0}%
(g_{t}+\mathrm{I}_{L\cap\operatorname*{dom}f})\\
&\leq &\max\left\{  \limsup
\limits_{\gamma_{t}\rightarrow\gamma,f_{t}(\theta)\rightarrow0}(f_{t}%
+\mathrm{I}_{L\cap\operatorname*{dom}f}),-1\right\}. \label{be}
\end{eqnarray}
But the two functions on the left and the right have the same value $0$
at $\theta,$ and so 
\begin{align*}
\partial(g_{\gamma}+\mathrm{I}_{L\cap\operatorname*{dom}f})(\theta)  &
\subset\partial\left(  \max\left\{  \limsup\limits_{\gamma_{t}\rightarrow
\gamma,f_{t}(\theta)\rightarrow0}f_{t}+\mathrm{I}_{L\cap\operatorname*{dom}%
f},-1\right\}  \right)  (\theta)\\
&  =\partial\left(  \limsup\limits_{\gamma_{t}\rightarrow\gamma,f_{t}%
(\theta)\rightarrow0}f_{t}+\mathrm{I}_{L\cap\operatorname*{dom}f}\right)
(\theta)=\partial\left(  \tilde{f}_{\gamma}+\mathrm{I}_{L\cap
\operatorname*{dom}f}\right)  (\theta),
\end{align*}
where the first equality comes from Proposition \ref{thmcompact0} applied to
the finite family $\left\{  \tilde{f}_{\gamma},-1\right\}  .$ Finally, the
desired inclusion follows thanks to (\ref{im}) and (\ref{tta}).
\end{proof}
\ Let us introduce a function which asigns to each given $\gamma\in
\widehat{T}(x)$ a net $(t_{i}^{\gamma})_{i}\subset T$ such that
\begin{equation}
\gamma_{t_{i}^{\gamma}}\rightarrow\gamma\text{, }f_{t_{i}^{\gamma}%
}(x)\rightarrow f(x). \label{selection}%
\end{equation}
Then, according to\ (\ref{be}),
\[
\lim_{\gamma_{t}\rightarrow\gamma}(g_{t}+\mathrm{I}_{\operatorname*{dom}%
f})=\lim\limits_{i}(g_{t_{i}^{\gamma}}+\mathrm{I}_{L\cap\operatorname*{dom}%
f})\leq\max\left\{  \limsup\limits_{i}(f_{t_{i}^{\gamma}}+\mathrm{I}%
_{L\cap\operatorname*{dom}f}),-1\right\}  ,
\]
and we obtain, reasoning as above,
\begin{equation}
\partial f(x)=\bigcap\nolimits_{L\in\mathcal{F}(x)}\operatorname*{co}\left\{
\bigcup\nolimits_{\gamma\in\widehat{T}(x)}\partial(\limsup_{i}f_{t_{i}%
^{\gamma}}+\mathrm{I}_{L\cap\operatorname*{dom}f})(x)\right\}  .
\label{mainformulaalternatived}%
\end{equation}

The use of the functions $g_{t}$ allows us to formulate $\partial f(x)$
involving only limits instead of upper limits. In fact, from\ (\ref{im}),
(\ref{tta}) and (\ref{be1}) we get
\begin{equation}
\partial f(x)=\bigcap\nolimits_{L\in\mathcal{F}(x)}\operatorname*{co}\left\{
\bigcup\nolimits_{\gamma\in\widehat{T}(x)}\partial\left(\lim\limits_{\gamma
_{t}\rightarrow\gamma,f_{t}(x)\rightarrow f(x)}(g_{t}+\mathrm{I}%
_{L\cap\operatorname*{dom}f})\right)(x)\right\}  . \label{fcor1}%
\end{equation}

\begin{cor}
\label{compactv}Suppose that the function\ $f$ is finite and continuous somewhere. Then, for every $x\in X,$%
\begin{align}
\partial f(x)  &  =\overline{\operatorname*{co}}\left\{  \bigcup
\nolimits_{\gamma\in\widehat{T}(x)}\partial(\limsup\nolimits_{i}f_{t_{i}^{\gamma}%
})(x)\right\}  +\mathrm{N}_{\operatorname*{dom}f}(x)\label{fc1}\\
&  =\operatorname*{co}\left\{  \bigcup\nolimits_{\gamma\in\widehat{T}%
(x)}\partial(\limsup\nolimits_{i}f_{t_{i}^{\gamma}})(x)\right\}  +\mathrm{N}%
_{\operatorname*{dom}f}(x)\text{ \ (if }X=\mathbb{R}^{n}), \label{fc2}%
\end{align}
where $(t_{i}^{\gamma})$ is defined\ in (\ref{selection}).
\end{cor}

\begin{proof}
Suppose, without loss of generality, that $x=\theta$ and $f(\theta)=0.$ According to
(\ref{mainformulaalternatived}), and using (\ref{MR}),%
\begin{align*}
\partial f(\theta)  &  =\bigcap\nolimits_{L\in\mathcal{F}(\theta
)}\operatorname*{co}\left\{  \bigcup\nolimits_{\gamma\in\widehat{T}(\theta
)}\partial(\limsup\nolimits_{i}f_{t_{i}^{\gamma}}+\mathrm{I}_{L\cap\operatorname*{dom}%
f})(\theta)\right\} \\
&  =\bigcap\nolimits_{L\in\mathcal{F}(\theta)}\left(  \operatorname*{co}%
\left\{  \bigcup\nolimits_{\gamma\in\widehat{T}(\theta)}\partial(\limsup\nolimits
_{i}f_{t_{i}^{\gamma}})(\theta)\right\}  +\mathrm{N}_{\operatorname*{dom}%
f}(\theta)+L^{\perp}\right)  ,
\end{align*}
and (\ref{fc2}) follows. To prove (\ref{fc1}) we first
obtain, due to the last relation and (\ref{tr}),%
\begin{equation}
\partial f(\theta)\subset\operatorname*{cl}\left(  A+B\right)  =\partial
\sigma_{A+B}(\theta)=\partial(\sigma_{A}+\sigma_{B})(\theta), \label{ba}%
\end{equation}
where $A:=\operatorname*{co}\left\{  \bigcup\nolimits_{\gamma\in
\widehat{T}(\theta)}\partial(\limsup_{i}f_{t_{i}^{\gamma}})(\theta)\right\}  $
and $B:=\mathrm{N}_{\operatorname*{dom}f}(\theta).$

Since $\limsup_{i}f_{t_{i}^{\gamma}}\leq f$ and both functions coincide at
$\theta,$ we have $A\subset\partial f(\theta).$ There also exist\ $m\geq0,$
$x_{0}\in\operatorname*{dom}f$ and $\theta$-neighborhood $U\subset X$ such
that $f(x_{0}+y)\leq m$,\ for all $y\in U.$ Then
\begin{equation}
\sigma_{A}(x_{0}+y)\leq\sigma_{\partial f(\theta)}(x_{0}+y)\leq f(x_{0}+y)\leq
m\text{ \ for all }y\in U; \label{cont}%
\end{equation}
that is, $\sigma_{A}$ is continuous at $x_{0}.$ Consequently,
since\ $\mathrm{\sigma}_{B}(x_{0})\leq0,$ (\ref{ba}) and (\ref{MR}) entail
\[
\partial f(\theta)\subset\partial\sigma_{A}(\theta)+\partial\sigma_{B}%
(\theta)=\operatorname*{cl}(A)+B,
\]
and the inclusion \textquotedblleft$\subset$\textquotedblright\ in (\ref{fc1})
follows. The opposite inclusion is straightforward.
\end{proof}

The following corollary provides a
characterization of $\partial f(x)$ in terms only of the active original
functions $f_{t}$'s.

\begin{cor}
\label{thmcompact0refiement}Fix $x\in X.$ If\ for each net $(t_{i})_{i}\subset
T$ satisfying\ $f_{t_{i}}(x)\rightarrow f(x),$ there exist a subnet
$(t_{i_{j}})_{j}\subset T$ of $(t_{i})_{i}$ and an index\ $t\in T$ such that\
\begin{equation}
\limsup\nolimits_{j}f_{t_{i_{j}}}(z)\leq f_{t}(z)\text{ \ \ for all }%
z\in\operatorname*{dom}f, \label{hyp}%
\end{equation}
then we have
\[
\partial f(x)=\bigcap\nolimits_{L\in\mathcal{F}(x)}\operatorname*{co}\left\{
\bigcup\nolimits_{t\in T(x)}\partial(f_{t}+\mathrm{I}_{L\cap
\operatorname*{dom}f})(x)\right\}  .
\]

\end{cor}

\begin{proof}
Given any $\gamma\in\widehat{T}(x)$ such that $\gamma_{t_{i}}\rightarrow
\gamma$ and $f_{t_{i}}(x)\rightarrow f(x)$, for some net $(t_{i})_{i}\subset
T,$ we choose\ a subnet $(t_{i_{j}}^{\gamma})_{j}$ in (\ref{selection})
satisfying (\ref{hyp}) for a certain $t^{\gamma}\in T.$ Then $t^{\gamma}\in
T(x),$ taking into account (\ref{hyp}) with $z=x,$ and by
(\ref{mainformulaalternatived})
\begin{align*}
\partial f(x)  &  =\bigcap\nolimits_{L\in\mathcal{F}(x)}\operatorname*{co}%
\left\{  \bigcup\nolimits_{\gamma\in\widehat{T}(x)}\partial(\limsup\nolimits_{j}f_{t_{i_{j}}^{\gamma}}+\mathrm{I}_{L\cap\operatorname*{dom}f})(x)\right\}
\\
&  \subset\bigcap\nolimits_{L\in\mathcal{F}(x)}\operatorname*{co}\left\{
\bigcup\nolimits_{\gamma\in\widehat{T}(x)}\partial(f_{t^{\gamma}}%
+\mathrm{I}_{L\cap\operatorname*{dom}f})(x)\right\}  ,
\end{align*}
where the last inclusion holds as $\limsup_{j}f_{t_{i_{j}}^{\gamma}%
}+\mathrm{I}_{L\cap\operatorname*{dom}f}\leq f_{t^{\gamma}}+\mathrm{I}%
_{L\cap\operatorname*{dom}f},$ by (\ref{hyp}), and these two functions take the same value
at $x.$ The inclusion \textquotedblleft$\subset$\textquotedblright%
\ follows as we have shown that $t^{\gamma}\in T(x).$ The opposite inclusion
is immediate.
\end{proof}

\section{From continuous to non-continuous\label{Secsubd}}

In this\ section, we consider again a family $f_{t}:X\rightarrow
\overline{\mathbb{R}},$ $t\in T,$ of convex functions defined on $X,$ and\ the
supremum function $f:=\sup_{t\in T}f_{t}.$ Based on the results of\ the
previous section we provide characterizations of $\partial f(x)$
involving\ only the $f_{t}$'s and not the regularized ones, i.e, the $f_{\gamma}$'s. We
shall need the following technical\ lemmas. In what follows,
$\operatorname*{cl}^{s}$ stands for the strong topology on $X^{\ast}$ (usually
denoted by $\beta(X^{\ast},X)).$

\begin{lem}
\label{lemteco}Assume that the convex functions\ $f_{t},$ $t\in T,$ are
proper, lsc, and\ such that $f_{\mid\operatorname*{aff}(\operatorname*{dom}%
f)}$ is continuous on $\operatorname*{ri}(\operatorname*{dom}f),$ assumed nonempty. Let $x\in\operatorname*{dom}f$ and the net $(z_{i}^{\ast})_{i\in
I}\subset X^{\ast}$ such that
\begin{equation}
\lim_{i}(\left\langle z_{i}^{\ast},x\right\rangle -\inf\nolimits_{t\in T}%
f_{t}^{\ast}(z_{i}^{\ast}))=f(x), \label{cd1}%
\end{equation}
and\ for all $z\in\operatorname*{dom}f$\
\begin{equation}
\limsup_{i}\left(  \left\langle z_{i}^{\ast},z\right\rangle -\inf
\nolimits_{t\in T}f_{t}^{\ast}(z_{i}^{\ast})\right)  >-\infty. \label{cd2}%
\end{equation}
Then, there exist a subnet $(z_{i_{j}}^{\ast})_{j}$ of $(z_{i}^{\ast})_{i}$ and
$z^{\ast}\in X^{\ast}$ such that
\begin{equation}
z^{\ast}\in\operatorname*{cl}\left(  \bigcup\nolimits_{t\in T_{\varepsilon
}(x)}\partial_{\varepsilon}f_{t}(x)+(\operatorname*{aff}(\operatorname*{dom}%
f))^{\perp}\right)  ,\text{ for all }\varepsilon>0, \label{gf00}%
\end{equation}
and 
\begin{equation}
\left\langle z_{i_{j}}^{\ast}-z^{\ast},z\right\rangle \rightarrow_{j}0  ,\text{ for all }z\in\operatorname*{aff}(\operatorname*{dom}f).
\label{gf}%
\end{equation}
In particular, if $\operatorname*{dom}f$ is finite-dimensional, then
(\ref{gf00}) also holds with $\operatorname*{cl}^{s}$ instead of
$\operatorname*{cl}.$
\end{lem}

\begin{proof}
We may assume that $x=\theta$ and $f(\theta)=0,$ and\ denote
$E:=\operatorname*{aff}(\operatorname*{dom}f)$ which is a closed subspace
with\ dual $E^{\ast}.$ We also denote $h:=\inf_{t\in T}f_{t}^{\ast},$ so that
(see (\ref{moreau}))%
\begin{equation}
h^{\ast}=(\inf_{t\in T}f_{t}^{\ast})^{\ast}=\sup_{t\in T}f_{t}^{\ast\ast}%
=\sup_{t\in T}f_{t}=f, \label{esea}%
\end{equation}
and
\begin{equation}
h^{\ast}(\theta)+h(z_{i}^{\ast})=f(\theta)+h(z_{i}^{\ast})=h(z_{i}^{\ast
})\rightarrow0. \label{irre}%
\end{equation}
Hence, for every fixed $\varepsilon>0,$ there is some $i_{0}\in I$ such
that\ for all $i\succeq i_{0}$
\begin{equation}
h^{\ast}(\theta)+h(z_{i}^{\ast})=\sup_{t\in T}f_{t}(\theta)+\inf_{t\in T}%
f_{t}^{\ast}(z_{i}^{\ast})=h(z_{i}^{\ast})<\varepsilon, \label{ese}%
\end{equation}
and so\
\begin{equation}
(z_{i}^{\ast})_{i\succeq i_{0}}\subset\partial_{\varepsilon}h^{\ast}%
(\theta)=\partial_{\varepsilon}f(\theta). \label{hy}%
\end{equation}
Now, using the continuity assumption, we choose $x_{0}\in\operatorname*{dom}%
f,$ a $\theta$-neighborhood $U\subset X$ and $r\geq0$ such that
\begin{equation}
f(x_{0}+y)\leq r\text{ \ for all }y\in U\cap E, \label{cones}%
\end{equation}
and, by (\ref{cd2}) with\ $z=x_{0}$ and (\ref{irre}),
\[
\limsup_{i}\left\langle z_{i}^{\ast},x_{0}\right\rangle >-\infty.
\]
Therefore we may assume, up to some subnet, that $\inf_{i}\left\langle
z_{i}^{\ast},x_{0}\right\rangle >-\infty$ and, so, by (\ref{hy}) and
(\ref{cones}), there is some $m>0$ such that%
\begin{equation}
\left\langle z_{i}^{\ast},y\right\rangle \leq f(x_{0}+y)+\varepsilon-\inf
_{i}\left\langle z_{i}^{\ast},x_{0}\right\rangle \leq m,\text{ for all }y\in
U\cap E\text{ and for all }i; \label{zr}%
\end{equation}
that is $(z_{i}^{\ast})_{i}\subset(U\cap E)^{\circ}.$ Since the last set is
weak*-compact in $E^{\ast},$ by the Alaoglu-Banach-Bourbaki theorem, there
exists a subnet $(z_{i_{j}\mid E}^{\ast})_{j}$ and $\tilde{z}^{\ast}\in
E^{\ast}$ such that
\begin{equation}
\left\langle z_{i_{j}\mid E}^{\ast}-\tilde{z}^{\ast},u\right\rangle
\rightarrow_{j}0\text{ \ for all }u\in E. \label{coc0}%
\end{equation}
Moreover, by the Hahn-Banach theorem, $\tilde{z}^{\ast}\in E^{\ast}$ is extended
to some $z^{\ast}\in X^{\ast},$ which satisfies
\begin{equation}
\left\langle z_{i_{j}}^{\ast}-z^{\ast},u\right\rangle =\left\langle
z_{i_{j}\mid E}^{\ast}-\tilde{z}^{\ast},u\right\rangle \rightarrow_{j}0\text{
\ for all }u\in E. \label{coc}%
\end{equation}
\ Now, using (\ref{ese}), we see that for each $i$ there exists $t_{i}\in T$ such that\
\[
f_{t_{i}}(\theta)+f_{t_{i}}^{\ast}(z_{i}^{\ast})\leq f_{t_{i}}^{\ast}%
(z_{i}^{\ast})<\varepsilon,
\]
entailing\ that $z_{i}^{\ast}\in\partial_{\varepsilon}f_{t_{i}}(\theta)$ and
\[
-f_{t_{i}}(\theta)=\left\langle z_{i}^{\ast},\theta\right\rangle -f_{t_{i}%
}(\theta)\leq f_{t_{i}}^{\ast}(z_{i}^{\ast})<\varepsilon;
\]
that is, $t_{i}\in T_{\varepsilon}(\theta)$ and so, 
$$
z_{i}^{\ast}\in
\bigcup\nolimits_{t\in T_{\varepsilon}(\theta)}\partial_{\varepsilon}%
f_{t}(\theta).
$$
We fix a\ weak* (strong, when $\operatorname*{dom}f$ is
finite-dimensional)\ $\theta$-neighborhood $V\subset X^{\ast}.$ Since
$E^{\ast}\ $is isomorphic to the quotient space $X_{\diagup E^{\perp}}^{\ast
},$ then $V_{\mid E}:=\left\{  u_{\mid E}^{\ast}:u^{\ast}\in V\right\}
\in\mathcal{N}_{E^{\ast}}$ (\cite{Fa01}), where $u_{\mid E}^{\ast}$ denotes the restriction of $u^{\ast}$ to $E^{\ast}$. 
Consequently, writing
\[
z_{i_{j}\mid E}^{\ast}\in A:=\left\{  u_{\mid E}^{\ast}\in E^{\ast}:u^{\ast
}\in\bigcup\nolimits_{t\in T_{\varepsilon}(\theta)}\partial_{\varepsilon}%
f_{t}(\theta)\right\}  ,
\]
and\ passing to the limit on $j,$ (\ref{coc}) leads us to%
\begin{equation}
z_{\mid E}^{\ast}\in A+V_{\mid E}. \label{pl}%
\end{equation}
In other words, there are $u^{\ast}\in\bigcup\nolimits_{t\in T_{\varepsilon
}(\theta)}\partial_{\varepsilon}f_{t}(\theta)$ and $v^{\ast}\in V$ such that
$z_{\mid E}^{\ast}=u_{\mid E}^{\ast}+v_{\mid E}^{\ast};$ that is,
\[
\left\langle z^{\ast},u\right\rangle =\left\langle u^{\ast}+v^{\ast
},u\right\rangle \text{ \ for all }u\in E,
\]
implying\ that
\[
z^{\ast}\in u^{\ast}+v^{\ast}+E^{\perp}\subset\bigcup\nolimits_{t\in
T_{\varepsilon}(\theta)}\partial_{\varepsilon}f_{t}(\theta)+E^{\perp}+V.
\]
The conclusion follows then by intersecting over $V$ and, after, over
$\varepsilon>0.$
\end{proof}

In the currrent framework, $\widehat{X^{\ast}}$ is the Stone-\v{C}ech
compactification of $X^{\ast}$, with respect to\ the discrete topology,
and\ the mappings $\gamma_{z^{\ast}}:\left[  0,1\right]  ^{X^{\ast}%
}\rightarrow\left[  0,1\right]  ,$ $z^{\ast}\in X^{\ast},$ are defined as in
(\ref{eval}), so that the convergence $\gamma_{z_{i}^{\ast}}\rightarrow\gamma$
for for a net $(z_{i}^{\ast})_{i}\subset X^{\ast}$ and $\gamma\in
\widehat{X^{\ast}}$ means%
\[
\varphi(z_{i}^{\ast})\rightarrow\gamma(\varphi)\text{ \ \ for all }\varphi
\in\left[  0,1\right]  ^{X^{\ast}}.
\]

\begin{lem}
\label{lemtec}Assume in Lemma \ref{lemteco} that the net $(\gamma_{z_{i}%
^{\ast}})_{i}$ converges in $\widehat{X^{\ast}}.$ Then for the function\
\[
\psi(z):=\limsup_{i}\left(\left\langle z_{i}^{\ast},z\right\rangle -\inf_{t\in
T}f_{t}^{\ast}(z_{i}^{\ast})+\mathrm{I}_{\operatorname*{dom}f}(z)\right),\text{
}z\in X,
\]
we have\
\begin{align*}
\partial\psi(x)  &  \subset\mathrm{N}_{\operatorname*{dom}f}(x)+\bigcap
\nolimits_{\varepsilon>0}\operatorname*{cl}\left(  \bigcup\nolimits_{t\in
T_{\varepsilon}(x)}\partial_{\varepsilon}f_{t}(x)+(\operatorname*{aff}%
(\operatorname*{dom}f))^{\perp}\right) \\
&  \subset\bigcap\nolimits_{\varepsilon>0}\operatorname*{cl}\left(
\bigcup\nolimits_{t\in T_{\varepsilon}(\theta)}\partial_{\varepsilon}%
f_{t}(\theta)+\mathrm{N}_{\operatorname*{dom}f}(\theta)\right),
\end{align*}
with $\operatorname*{cl}^{s}$ instead of $\operatorname*{cl}$ when
$\operatorname*{dom}f$ is finite-dimensional.
\end{lem}

\begin{proof}
We may suppose that $x=\theta$ and $f(\theta)=0.$ By Lemma \ref{lemteco}
there\ exist a subnet $(z_{i_{j}}^{\ast})_{j}$ of $(z_{i}^{\ast})_{i}$ and
\[
z^{\ast}\in\bigcap\nolimits_{\varepsilon>0}\operatorname*{cl}\left(
\bigcup\nolimits_{t\in T_{\varepsilon}(x)}\partial_{\varepsilon}%
f_{t}(x)+(\operatorname*{aff}(\operatorname*{dom}f))^{\perp}\right)
\]
such that $(z_{i_{j}}^{\ast})_{j}$ weak*-converges to $z^{\ast}$ in $E^{\ast}$
(where $E=\operatorname*{aff}(\operatorname*{dom}f)).$

We\ introduce the functions $g_{u^{\ast}}:X\rightarrow\mathbb{R}_{\infty},$
$u^{\ast}\in X^{\ast},$ defined as
\[
g_{u^{\ast}}:=\max\left\{  u^{\ast}-h(u^{\ast}),-1\right\}  ,
\]
where $h=\inf_{t\in T}f_{t}^{\ast}$ (already used in the proof of\ Lemma
\ref{lemteco}). Observe that (recall (\ref{esea}))
$$-1\leq g_{u^{\ast}}\leq\max\left\{  h^{\ast
},-1\right\}  =\max\left\{  f,-1\right\},
$$  and
\[
\varphi_{z}(u^{\ast}):=\frac{g_{u^{\ast}}(z)+1}{\max\left\{  f(z)+1,1\right\}
}\in\left[  0,1\right],  \text{ for all }z\in\operatorname*{dom}f.
\]
Hence, since $\varphi_{z}$ is obviously continuous on $X^{\ast}$ endowed with
the discrete topology, the convergence assumption of\ $(\gamma_{z_{i}^{\ast}%
})_{i}$ ensures that, for each $z\in\operatorname*{dom}f,$ the net
\[
\gamma_{z_{i}^{\ast}}(\varphi_{z})=\frac{g_{z_{i}^{\ast}}(z)+1}{\max\left\{
f(z)+1,1\right\}  }%
\]
also converges, as well as the net $(g_{z_{i}^{\ast}}(z))_{i}$. Then, taking into
account (\ref{cd1}) and (\ref{gf}), we obtain
\begin{align*}
\lim_{i}g_{z_{i}^{\ast}}(z)  &  =\lim_{i}\max\left\{  \left\langle z_{i}%
^{\ast},z\right\rangle -h(z_{i}^{\ast}),-1\right\} \\
&  =\lim_{j}\max\left\{  \left\langle z_{i_{j}}^{\ast},z\right\rangle
,-1\right\}  =\max\left\{  \left\langle z^{\ast},z\right\rangle ,-1\right\}  ,
\end{align*}
which gives
\[
\limsup_{i}\left\langle z_{i}^{\ast},z\right\rangle \leq\limsup_{i}%
(\max\{\left\langle z_{i}^{\ast},z\right\rangle ,-1\})=\max\{\left\langle
z^{\ast},z\right\rangle ,-1\}.
\]
But\ both functions $\limsup_{i}z_{i}^{\ast}+\mathrm{I}_{\operatorname*{dom}%
f}$ and $\max\{z^{\ast},-1\}+\mathrm{I}_{\operatorname*{dom}f}$ coincide at
$\theta,$ and so
\[
\partial\left(\limsup_{i}(z_{i}^{\ast}+\mathrm{I}_{\operatorname*{dom}f}%
)\right)(\theta)\subset\partial(\max\{z^{\ast}+\mathrm{I}_{\operatorname*{dom}%
f},-1\})(\theta),
\]
and (\ref{mainformulaalternative}) applied to the (finite) family $\{z^{\ast
}+\mathrm{I}_{\operatorname*{dom}f},-1\}$ yields (recall (\ref{cd1}))
\begin{align*}
\partial\psi(\theta)  &  =\partial\left(\limsup_{i}(z_{i}^{\ast}+\mathrm{I}%
_{\operatorname*{dom}f})\right)(\theta)\\
&  \subset z^{\ast}+\mathrm{N}_{\operatorname*{dom}f}(\theta).\\
&  \subset\mathrm{N}_{\operatorname*{dom}f}(\theta)+\bigcap
\nolimits_{\varepsilon>0}\operatorname*{cl}\left(  \bigcup\nolimits_{t\in
T_{\varepsilon}(\theta)}\partial_{\varepsilon}f_{t}(\theta
)+(\operatorname*{aff}(\operatorname*{dom}f))^{\perp}\right) \\
&  \subset\bigcap\nolimits_{\varepsilon>0}\operatorname*{cl}\left(
\bigcup\nolimits_{t\in T_{\varepsilon}(\theta)}\partial_{\varepsilon}%
f_{t}(\theta)+\mathrm{N}_{\operatorname*{dom}f}(\theta)\right)  .
\end{align*}

\end{proof}

\begin{theo}
\label{hamdkathirbisb}Let $f_{t}:X\rightarrow\overline{\mathbb{R}},$ $t\in T,$
be\ convex functions and $f=\sup_{t\in T}f_{t}.$ Then, for every $x\in X,$\
\begin{equation}
\partial f(x)=\bigcap\nolimits_{L\in\mathcal{F}(x)}\operatorname*{co}\left\{
\bigcap\nolimits_{\varepsilon>0}\operatorname*{cl}\nolimits^{s}\left(
\bigcup\nolimits_{t\in T_{\varepsilon}(x)}\partial_{\varepsilon}%
(f_{t}+\mathrm{I}_{L\cap\operatorname*{dom}f})(x)\right)  \right\}  .
\label{f1}%
\end{equation}
If, in addition,
\begin{equation}
\operatorname*{cl}f=\sup_{t\in T}(\operatorname*{cl}f_{t}), \label{closure}%
\end{equation}
then
\begin{equation}
\partial f(x)=\bigcap\nolimits_{L\in\mathcal{F}(x)}\operatorname*{co}\left\{
\bigcap\nolimits_{\varepsilon>0}\operatorname*{cl}\left(  \bigcup
\nolimits_{t\in T_{\varepsilon}(x)}\partial_{\varepsilon}f_{t}(x)+\mathrm{N}%
_{L\cap\operatorname*{dom}f}(x)\right)  \right\}  . \label{f1b}%
\end{equation}

\end{theo}
\begin{rem}(before the proof) \emph{Formula\ (\ref{f1}) leads straightforwardly\ to the
following characterization of }$\partial f(x),$\emph{ using the strong
closure\ }%
\[
\partial f(x)=\bigcap\nolimits_{L\in\mathcal{F}(x),\varepsilon>0}%
\overline{\operatorname*{co}}^{s}\left\{  \bigcup\nolimits_{t\in
T_{\varepsilon}(x)}\partial_{\varepsilon}(f_{t}+\mathrm{I}_{L\cap
\operatorname*{dom}f})(x)\right\}  ,
\]
\emph{improving the one of Proposition \ref{p2}, which is given in terms of the
weak*-closure. However, on\ despite that both formulas involve similar
elements, the order in taking the intersection over }$\varepsilon$\emph{ leads
to different interpretations of }$\partial f(x).$\emph{ For instance, if
}$T$\emph{ is finite, }$T=T(x)$\emph{ and }$f$\emph{ is continuous, then
(\ref{f1}) reads\ }%
\[
\partial f(x)=\operatorname*{co}\left\{  \bigcup\nolimits_{t\in T(x)}\partial
f_{t}(x)\right\}  ,
\]
\emph{giving Valadier's formula (see, e.g., \cite{Va69}), while Proposition \ref{p2}
yields }%
\[
\partial f(x)=\bigcap\nolimits_{\varepsilon>0}\overline{\operatorname*{co}%
}\left\{  \bigcup\nolimits_{t\in T(x)}\partial_{\varepsilon}f_{t}(x)\right\}
,
\]
\emph{which turns out to be the Br{\o}ndsted formula (\cite{Br72}; see, also,
\cite[Corollary 12]{HLZ08}).}
\end{rem}
\begin{proof}
The inclusions \textquotedblleft$\supset$\textquotedblright\ in both formulas
are straightforward. We may suppose, without loss of generality, that $x=\theta,$ $f(\theta)=0$
and $\partial f(\theta)\neq\emptyset;$ hence,
\begin{equation}
\partial(\operatorname*{cl}f)(\theta)=\partial f(\theta)\text{ \ and
\ }f(\theta)=(\operatorname*{cl}f)(\theta)=0. \label{egz}%
\end{equation}
We proceed in three steps:\newline Step 1. We assume that\ all the $f_{t}$'s
are proper and lsc; hence, (\ref{closure}) obviously holds. We fix
$L\in\mathcal{F}(\theta),$ and define\ the functions
\begin{equation}
\tilde{f}_{t}:=f_{t}+\mathrm{I}_{L},\text{ }t\in T,\text{\ \ and \ \ }%
h:=\inf\nolimits_{t\in T}\tilde{f}_{t}^{\ast}. \label{hinf}%
\end{equation}
The $\tilde{f}_{t}$'s are proper and lsc, and we have (see (\ref{moreau}))
\begin{equation}
(f+\mathrm{I}_{L})(z)=\sup\nolimits_{t\in T}\tilde{f}_{t}(z)=\sup
\nolimits_{t\in T}\tilde{f}_{t}^{\ast\ast}(z)=(\inf\nolimits_{t\in T}\tilde
{f}_{t}^{\ast})^{\ast}(z)=h^{\ast}(z); \label{rw}%
\end{equation}
that is,
\[
(f+\mathrm{I}_{L})(z)=\sup\left\{  \left\langle z,z^{\ast}\right\rangle
-h(z^{\ast}),\text{ }z^{\ast}\in X^{\ast}\right\}  ,
\]
and\ (\ref{mainformulaalternatived}) applied with $T=X^{\ast}$ (endowed with
the discrete topology) yields
\begin{equation}
\partial(f+\mathrm{I}_{L})(\theta)\subset\operatorname*{co}\left\{
\bigcup\nolimits_{\gamma\in\widehat{X^{\ast}}(\theta)}\partial\left(\limsup
_{i}(z_{i}^{\ast^{\gamma}}-h(z_{i}^{\ast^{\gamma}})+\mathrm{I}_{L\cap
\operatorname*{dom}f})\right)(\theta)\right\}  , \label{no}%
\end{equation}
where $\widehat{X^{\ast}}(\theta)$ repesents the set $\widehat{T}(\theta)$
given in (\ref{dgam}); that is,
\[
\widehat{X^{\ast}}(\theta)=\left\{  \gamma\in\widehat{X^{\ast}}:
\limsup_{\gamma_{z^{\ast}}\rightarrow\gamma}(-h(z^{\ast}))=0\right\}  ,
\]
and $(z_{i}^{\ast^{\gamma}})_{i}\subset X^{\ast}$ is a fixed net such that
$\gamma_{z_{i}^{\ast^{\gamma}}}\rightarrow\gamma$ and $h(z_{i}^{\ast^{\gamma}%
})\rightarrow0$ (by (\ref{selection})). Consequenlty, for every $\gamma
\in\widehat{X^{\ast}}(\theta)$, Lemma \ref{lemtec} applies and yields
\begin{align}
\nonumber\partial\left(\limsup_{i}(z_{i}^{\ast^{\gamma}}-h(z_{i}^{\ast^{\gamma}}%
)+\mathrm{I}_{L\cap\operatorname*{dom}f})\right)(\theta)\qquad\qquad\qquad\qquad
\\
\subset
\bigcap
\nolimits_{\varepsilon>0}\operatorname*{cl}\nolimits^{s}\left(  \bigcup
\nolimits_{t\in T_{\varepsilon}^{1}(\theta)}\partial_{\varepsilon}\tilde
{f}_{t}(\theta)+\mathrm{N}_{L\cap\operatorname*{dom}f}(\theta)\right), \label{le}
\end{align}
where%
\begin{equation}
T_{\varepsilon}^{1}(\theta):=\left\{  t\in T:\tilde{f}_{t}(\theta
)\geq-\varepsilon\right\}  =T_{\varepsilon}(\theta). \label{tt}%
\end{equation}
Indeed, condition (\ref{cd2}) is satisfied when the left-hand side in
(\ref{le}) is nonempty, and thus the function $\limsup_{i}(z_{i}^{\ast
^{\gamma}}-h(z_{i}^{\ast^{\gamma}})+\mathrm{I}_{L\cap\operatorname*{dom}f})$
is proper. Consequently, combining (\ref{no}), (\ref{le}) and (\ref{tt}),
\begin{equation}
\partial(f+\mathrm{I}_{L})(\theta)\subset\operatorname*{co}\left\{
\bigcap\nolimits_{\varepsilon>0}\operatorname*{cl}\nolimits^{s}\left(
\bigcup\nolimits_{t\in T_{\varepsilon}(\theta)}\partial_{\varepsilon}\tilde
{f}_{t}(\theta)+\mathrm{N}_{L\cap\operatorname*{dom}f}(\theta)\right)
\right\}  , \label{non}%
\end{equation}
and the inclusion \textquotedblleft$\subset$\textquotedblright\ in (\ref{f1})
follows since 
$\partial f(\theta)\subset\partial(f+\mathrm{I}_{L})(\theta)$
and 
$$\partial_{\varepsilon}\tilde{f}_{t}(\theta)+\mathrm{N}_{L\cap
\operatorname*{dom}f}(\theta)\subset\partial_{\varepsilon}(f_{t}%
+\mathrm{I}_{L\cap\operatorname*{dom}f})(\theta).
$$
Moreover, due to\ the fact that $\partial_{\varepsilon}\tilde{f}_{t}%
(\theta)\subset\operatorname*{cl}(\partial_{\varepsilon}f_{t}(\theta
)+L^{\perp})$ (see, e.g., \cite{HP93}), (\ref{non}) implies that
\begin{align}
\partial f(\theta)  &  \subset\operatorname*{co}\left\{  \bigcap
\nolimits_{\varepsilon>0}\operatorname*{cl}\nolimits^{s}\left(  \bigcup
\nolimits_{t\in T_{\varepsilon}(\theta)}\operatorname*{cl}(\partial
_{\varepsilon}f_{t}(\theta)+L^{\perp})+\mathrm{N}_{L\cap\operatorname*{dom}%
f}(\theta)\right)  \right\} \nonumber\\
&  \subset\operatorname*{co}\left\{  \bigcap\nolimits_{\varepsilon
>0}\operatorname*{cl}\nolimits^{s}\left(  \operatorname*{cl}\left(
\bigcup\nolimits_{t\in T_{\varepsilon}(\theta)}\partial_{\varepsilon}%
f_{t}(\theta)+\mathrm{N}_{L\cap\operatorname*{dom}f}(\theta)\right)  \right)
\right\} \nonumber\\
&  =\operatorname*{co}\left\{  \bigcap\nolimits_{\varepsilon>0}%
\operatorname*{cl}\left(  \bigcup\nolimits_{t\in T_{\varepsilon}(\theta
)}\partial_{\varepsilon}f_{t}(\theta)+\mathrm{N}_{L\cap\operatorname*{dom}%
f}(\theta)\right)  \right\}  , \label{fs1}%
\end{align}
which yields the inclusion \textquotedblleft$\subset$\textquotedblright\ in
(\ref{f1b}).
\vspace{.3cm}
\newline Step 2. We suppose\ that (\ref{closure}) holds and we fix
$L\in\mathcal{F}(\theta).$ By (\ref{egz}) we choose a $\theta$-neighborhood
$U\subset X$ such that
\begin{equation}
f(z)\geq(\operatorname*{cl}f)(z)\geq-1,\text{ for all }z\in U, \label{lss}%
\end{equation}
and denote $S:=\left\{  t\in T:~\operatorname*{cl}f_{t}\text{ is
proper}\right\}  .$ We define the functions
\[
g_{t}:=\operatorname*{cl}f_{t},\text{ if }t\in S\text{, and }g_{t}%
:=\max\left\{  \operatorname*{cl}f_{t},-1\right\}  ,\text{ otherwise.}%
\]
Then (see the proof of\ \cite[Theorem 4]{HLZ08}, page 871) $g_{t}$ is proper, lsc and 
convex,%
\[
g(z):=\sup_{t\in T}g_{t}(z)=(\operatorname*{cl}f)(z),\ \text{for all }z\in U;
\]
hence, $g(\theta)=0,$%
\[
\left\{  t\in T:g_{t}(\theta)\geq-\varepsilon\right\}  \subset T_{\varepsilon
}(\theta)\cap S,\text{ }\forall\varepsilon\in\left]  0,1\right[  ,
\]%
\[
\partial_{\varepsilon}g_{t}(\theta)\subset\partial_{2\varepsilon}f_{t}%
(\theta),\text{ }\partial_{\varepsilon}(g_{t}+\mathrm{I}_{L\cap
\operatorname*{dom}f})(\theta)\subset\partial_{2\varepsilon}(f_{t}%
+\mathrm{I}_{L\cap\operatorname*{dom}f})(\theta),\text{ }\forall\varepsilon
\in\left]  0,1\right[  ,
\]
and\
\begin{equation}
\partial f(\theta)=\partial(\operatorname*{cl}f)(\theta)=\partial g(\theta).
\label{gh}%
\end{equation}
Consequently, by Step 1,%
\begin{align*}
\partial f(\theta)=\partial g(\theta)  &  =\bigcap\nolimits_{L\in
\mathcal{F}(\theta)}\operatorname*{co}\left\{  \bigcap\nolimits_{\varepsilon
>0}\operatorname*{cl}\nolimits^{s}\left(  \bigcup\nolimits_{t\in T,\text{
}g_{t}(\theta)\geq-\varepsilon}\partial_{\varepsilon}(g_{t}+\mathrm{I}%
_{L\cap\operatorname*{dom}g})(\theta)\right)  \right\} \\
&  \subset\bigcap\nolimits_{L\in\mathcal{F}(\theta)}\operatorname*{co}\left\{
\bigcap\nolimits_{\varepsilon>0}\operatorname*{cl}\nolimits^{s}\left(
\bigcup\nolimits_{t\in T_{\varepsilon}(\theta)}\partial_{2\varepsilon}%
(f_{t}+\mathrm{I}_{L\cap\operatorname*{dom}f})(\theta)\right)  \right\}  ,
\end{align*}
entailing the desired inclusion \textquotedblleft$\subset$%
\textquotedblright\ in (\ref{f1}). 

Similarly, (\ref{fs1}) yields
\begin{align}
\partial f(\theta)  &  =\bigcap\nolimits_{L\in\mathcal{F}(\theta
)}\operatorname*{co}\left\{  \bigcap\nolimits_{0<\varepsilon<1}%
\operatorname*{cl}\left(  \bigcup\nolimits_{t\in T,\text{ }g_{t}(\theta
)\geq-\varepsilon}\partial_{\varepsilon}g_{t}(\theta)+\mathrm{N}%
_{L\cap\operatorname*{dom}g}(\theta)\right)  \right\} \nonumber\\
&  \subset\bigcap\nolimits_{L\in\mathcal{F}(\theta)}\operatorname*{co}\left\{
\bigcap\nolimits_{0<\varepsilon<1}\operatorname*{cl}\left(  \bigcup
\nolimits_{t\in T_{\varepsilon}(\theta)}\partial_{2\varepsilon}f_{t}%
(\theta)+\mathrm{N}_{L\cap\operatorname*{dom}f}(\theta)\right)  \right\}  ,
\label{rec1}%
\end{align}
which easily leads to the inclusion \textquotedblleft$\subset$%
\textquotedblright\ in (\ref{f1b}).\newline Step 3. We prove (\ref{f1}) in the
general case, without assuming (\ref{closure}). We fix $L\in\mathcal{F}%
(\theta)$ and define\
\[
\hat{f}_{t}:=f_{t}+\mathrm{I}_{L\cap\operatorname*{dom}f},
\]
so that
\[
f_{L}:=\sup_{t\in T}\hat{f}_{t}=f+\mathrm{I}_{L\cap\operatorname*{dom}%
f}=f+\mathrm{I}_{L},
\]%
\[
\hat{f}_{t}(\theta)=f_{t}(\theta),\text{ }f_{L}(\theta)=0,\text{ and
}\operatorname*{dom}f_{L}=L\cap\operatorname*{dom}f.
\]
Moreover, the family $\left\{  \hat{f}_{t},\text{ }t\in T\right\}  $ satisfies
condition (\ref{closure}) (see the proof of Proposition \ref{p2}). Since
(see\ \cite[Lemma 3.1]{CHL19})
\[
\partial f(\theta)=\bigcap\nolimits_{L\in\mathcal{F}(\theta)}\partial
(f+\mathrm{I}_{L})(\theta)=\bigcap\nolimits_{L\in\mathcal{F}(\theta)}\partial
f_{L}(\theta),
\]
applying\ Step 2 to the family $\left\{  \hat{f}_{t},\text{ }t\in T\right\}  $
we get
\begin{align*}
\partial f(\theta)  &  =\bigcap\nolimits_{L\in\mathcal{F}(\theta)}\partial
f_{L}(\theta)\\
&  \subset\bigcap\nolimits_{L\in\mathcal{F}(\theta)}\operatorname*{co}\left\{
\bigcap\nolimits_{\varepsilon>0}\operatorname*{cl}\nolimits^{s}\left(
\bigcup\nolimits_{t\in T,\text{ }\hat{f}_{t}(\theta)\geq-\varepsilon}%
\partial_{\varepsilon}(\hat{f}_{t}+\mathrm{I}_{L\cap\operatorname*{dom}f_{L}%
})(\theta)\right)  \right\} \\
&  =\bigcap\nolimits_{L\in\mathcal{F}(\theta)}\operatorname*{co}\left\{
\bigcap\nolimits_{\varepsilon>0}\operatorname*{cl}\nolimits^{s}\left(
\bigcup\nolimits_{T_{\varepsilon}(\theta)}\partial_{\varepsilon}%
(f_{t}+\mathrm{I}_{L\cap\operatorname*{dom}f})(\theta)\right)  \right\}  ,
\end{align*}
and the inclusion \textquotedblleft$\subset$\textquotedblright\ in (\ref{f1}) follows.
\end{proof}

\medskip

The following corollary closing this section considers a frequent hypothesis in
the literature.

\begin{cor}
\label{khay}Let $f_{t}:X\rightarrow\overline{\mathbb{R}},$ $t\in T,$
be\ convex functions. If $f=\sup_{t\in T}f_{t}$ is finite and continuous at
some point, then for every $x\in X$%
\begin{align*}
\partial f(x)  &  =\mathrm{N}_{\operatorname*{dom}f}(x)+\overline
{\operatorname*{co}}\left\{  \bigcap\nolimits_{\varepsilon>0}%
\operatorname*{cl}\left(  \bigcup\nolimits_{t\in T_{\varepsilon}(x)}%
\partial_{\varepsilon}f_{t}(x)\right)  \right\} \\
&  =\mathrm{N}_{\operatorname*{dom}f}(x)+\operatorname*{co}\left\{
\bigcap\nolimits_{\varepsilon>0}\operatorname*{cl}\left(  \bigcup
\nolimits_{t\in T_{\varepsilon}(x)}\partial_{\varepsilon}f_{t}(x)\right)
\right\}  \text{ \ \ (if }X=\mathbb{R}^{n}).
\end{align*}

\end{cor}

\begin{proof}
The proof is similar to the one of Theorem \ref{hamdkathirbisb}, but with
the use of the formulas in Corollary \ref{compactv} instead of formula
(\ref{mainformulaalternatived}).
\end{proof}

We close this section with an extension of Theorem \ref{hamdkathirbisb} to nonconvex functions. We also refer to \cite{MoNg13}, and references therein, for other studies on the subdifferential of the supremum of nonconvex functions. 
\begin{cor}
Let $f_{t}:X\rightarrow \overline{\mathbb{R}},$ $t\in T,$ be a family of
non-necessarily convex functions and $f:=\sup_{t\in T}f_{t}.$ Assume that 
\[
f^{\ast \ast }=\sup_{t\in T}f_{t}^{\ast \ast }.
\]
Then (\ref{f1b}) holds.
\end{cor}
\begin{proof}
It suffices to prove the inclusion \textquotedblleft $\subset $%
\textquotedblright\ in (57) for $x$ such that $\partial f(x)\neq \emptyset ;$
hence, $f^{\ast }$ is proper, $f(x)=f^{\ast \ast }(x)$ and $\partial
f(x)=\partial (\overline{\limfunc{co}}f)(x)=\partial f^{\ast \ast }(x).$
Thus, applying\ the second statement in Theorem 11 to the family $\left\{
f_{t}^{\ast \ast },\text{ }t\in T\right\} ,$ 
\[
\partial f(x)=\partial f^{\ast \ast }(x)=\bigcap\nolimits_{L\in \mathcal{F}%
(x)}\limfunc{co}\left\{ \bigcap\nolimits_{\varepsilon >0}\limfunc{cl}\left(
\bigcup\nolimits_{t\in T_{\varepsilon }^{1}(x)}\partial _{\varepsilon
}f_{t}^{\ast \ast }(x)+\mathrm{N}_{L\cap \limfunc{dom}f^{\ast \ast
}}(x)\right) \right\} ,
\]%
where $T_{\varepsilon }^{1}(x):=\left\{ t\in T:f_{t}^{\ast \ast }(x)\geq
f(x)-\varepsilon \right\} .$ Observe that every\ $t\in T_{\varepsilon
}^{1}(x)$ satisfies 
\[
f_{t}(x)\geq f_{t}^{\ast \ast }(x)\geq f(x)-\varepsilon \geq
f_{t}(x)-\varepsilon ;
\]%
hence, $t\in T_{\varepsilon }(x)$ and $\partial _{\varepsilon }f_{t}^{\ast
\ast }(x)\subset \partial _{2\varepsilon }f_{t}(x).$ Additionally, the
inequality $f^{\ast \ast }\leq f$ implies that $\mathrm{N}_{L\cap \limfunc{%
dom}f^{\ast \ast }}(x)\subset \mathrm{N}_{L\cap \limfunc{dom}f}(x),$ and the
desired inclusion follows.
\end{proof}

\section{Two\ applications in optimization\label{Secap}}

First, in this section, we apply the previous results\ to\ extend the classical
Fenchel duality to the nonconvex framework. This will lead us to recover some of
the results in \cite{CH10,CH12,CH13} (see, also, \cite{LV10}), relating the
solution set of a nonconvex optimization problem and its convexified
relaxation. Second, we establish\ Fritz-John and KKT optimality conditions for
convex semi-infinite optimization problems, improving similar results in \cite{CHL19}.

Given\ a function\ $g:X\rightarrow\mathbb{R}_{\infty}$, we recall that the Fenchel
conjugate\ of $g$ is the function\ $f:X^{\ast}\rightarrow\overline{\mathbb{R}%
}$, given by
\begin{equation}\label{eqs}
f(x^{\ast}):=\sup_{x\in X}(\left\langle x,x^{\ast}\right\rangle -g(x)).
\end{equation}
When $g$ is proper, lsc and convex, the classical Fenchel duality, together with (\ref{moreau}), yields
\begin{equation}
\partial f=(\partial g)^{-1}. \label{fenchel}%
\end{equation}
We extend this relation to non-necessarily convex functions. We denote below the closure with respect to the weak topology in $X$ by $\operatorname*{cl}^{w}$.

\begin{prop}\label{conj}
Assume that the function $f$ is proper. Then, for every $x^{\ast}\in X^{\ast},$
\[
\partial f(x^{\ast})=\bigcap\nolimits_{L\in\mathcal{F}(x^{\ast})}%
\operatorname*{co}\left\{  \bigcap\nolimits_{\varepsilon>0}\operatorname*{cl}%
\nolimits^{w}\left(  (\partial_{\varepsilon}g)^{-1}(x^{\ast})+\mathrm{N}%
_{L\cap\operatorname*{dom}f}(x^{\ast})\right)  \right\}  .
\]
If, in addition, $f$ is finite and (weak*-) continuous somewhere, then
\begin{align*}
\partial f(x^{\ast})  &  =\overline{\operatorname*{co}}\left\{  \left(
(\partial(\operatorname*{cl}\nolimits^{w} g))^{-1}(x^{\ast})\right)  \right\}  +\mathrm{N}%
_{\operatorname*{dom}f}(x^{\ast})\\
&  =\operatorname*{co}\left\{  \left(  (\partial(\operatorname*{cl} g))^{-1}(x^{\ast
})\right)  \right\}  +\mathrm{N}_{\operatorname*{dom}f}(x^{\ast})\text{ \ (if
}X=\mathbb{R}^{n}\text{),}%
\end{align*}
where $\operatorname*{cl}^{w} g$ is the weak-lsc hull of $g.$
\end{prop}

\begin{proof}
We define the convex functions $f_{x}:X^{\ast}\rightarrow\overline{\mathbb{R}%
},$ $x\in X,$ as
\[
f_{x}(x^{\ast}):=\left\langle x,x^{\ast}\right\rangle -g(x),\text{ }%
x\in\operatorname*{dom}g,
\]
so that $f_{x}$ are weak*-continuous and $f=\sup_{x\in\operatorname*{dom}%
g}f_{x}.$ Then, according to formula\ (\ref{f1b}), for every $x^{\ast}\in X^{\ast}%
$\ we have
\[
\partial f(x^{\ast})=\bigcap\nolimits_{L\in\mathcal{F}(x^{\ast})}%
\operatorname*{co}\left\{  \bigcap\nolimits_{\varepsilon>0}\operatorname*{cl}%
\nolimits^{w}\left(  \bigcup\nolimits_{x\in T_{\varepsilon}(x^{\ast})}%
\partial_{\varepsilon}f_{x}(x^{\ast})+\mathrm{N}_{L\cap\operatorname*{dom}%
f}(x^{\ast})\right)  \right\}  ,
\]
where
\[
T_{\varepsilon}(x^{\ast}):=\left\{  x\in\operatorname*{dom}g: f_{x}%
(x^{\ast})\geq f(x^{\ast})-\varepsilon\right\}  =(\partial_{\varepsilon
}g)^{-1}(x^{\ast}).
\]
Consequently, the first\ formula comes from\ the fact that\ $\partial
_{\varepsilon}f_{x}(x^{\ast})=\left\{  x\right\}  .$

Assume now that $f$ is finite and weak*-continuous somewhere. Then, arguing in a
similar way, but using Corollary \ref{khay} instead of (\ref{f1b}),
\begin{align*}
\partial f(x^{\ast})  &  =\overline{\operatorname*{co}}\left\{  \bigcap
\nolimits_{\varepsilon>0}\operatorname*{cl}\nolimits^{w}\left(  (\partial
_{\varepsilon}g)^{-1}(x^{\ast})\right)  \right\}  +\mathrm{N}%
_{\operatorname*{dom}f}(x^{\ast})\\
&  =\operatorname*{co}\left\{  \bigcap\nolimits_{\varepsilon>0}%
\operatorname*{cl}\left(  (\partial_{\varepsilon}g)^{-1}(x^{\ast
})\right)  \right\}  +\mathrm{N}_{\operatorname*{dom}f}(x^{\ast})\text{ \ (if
}X=\mathbb{R}^{n}\text{).}%
\end{align*}
The desired formulas follow as
\begin{equation}
\bigcap\nolimits_{\varepsilon>0}\operatorname*{cl}\nolimits^{w}\left(
(\partial_{\varepsilon}g)^{-1}(x^{\ast})\right)  =(\partial(\operatorname*{cl}\nolimits^{w} g))^{-1}(x^{\ast}), \label{yad}%
\end{equation}
according to \cite[Lemma 2.3]{CHG18}.
\end{proof}

Observing that $\operatorname{Argmin}(\overline{\operatorname*{co}}g)=\partial
f(\theta),$ the previous proposition gives:
\begin{cor}
Assume that the function $f$ is proper. Then we have
\[
\operatorname{Argmin}(\overline{\operatorname*{co}}g)=\bigcap\nolimits_{L\in
\mathcal{F}(\theta)}\operatorname*{co}\left\{  \bigcap\nolimits_{\varepsilon
>0}\operatorname*{cl}\nolimits^{w}\left(  \varepsilon\text{-}\operatorname{Argmin}%
g+\mathrm{N}_{L\cap\operatorname*{dom}f}(\theta)\right)  \right\}  .
\]
If, in addition, $f$ is finite and continuous at some point, then
\begin{align*}
\operatorname{Argmin}(\overline{\operatorname*{co}}g)  &  =\overline
{\operatorname*{co}}(\operatorname{Argmin}(\operatorname*{cl}\nolimits^{w} g))+\mathrm{N}%
_{\operatorname*{dom}f}(\theta)\\
&  =\operatorname*{co}(\operatorname{Argmin}(\operatorname*{cl} g))+\mathrm{N}%
_{\operatorname*{dom}f}(\theta)\text{ \ (if }X=\mathbb{R}^{n}\text{).}%
\end{align*}

\end{cor}

When $X$ is a normed space, the set $\partial f(x^*)$ is also seen as a subset of the bidual space, whereas Proposition \ref{conj} characterizes only the part of $\partial f(x^*)$ in the subspace $X$ of $X^{**}$. A light adaptation of Proposition \ref{conj} allows us to have a complete picture of $\partial f(x^*)$, as a proper set of the bidual space $X^{**}$. 
In such a setting, we denote the weak*-topology $\sigma(X^{**},X^*)$ in $X^{**}$ by $w^{**}$, and introduce the function $\overline{g}^{w^{**}}:X^{**}\to\overline{\mathbb{R}}$ defined by
$$
\overline{g}^{w^{**}}(y)=\liminf_{x\to^{w^{**}}y} g(x),\ y\in X^{**}.
$$
We refer, e.g., to \cite[Chapter 1]{BoSh00} for these concepts.

\begin{prop}Assume that $X$ is a normed space and $X^*$ is endowed with the dual norm topology. If the function $f$ is proper, then for every $x^{\ast}\in X^{\ast}$
\[
\partial f(x^{\ast})=\bigcap\nolimits_{L\in\mathcal{F}(x^{\ast})}%
\operatorname*{co}\left\{  \bigcap\nolimits_{\varepsilon>0}\operatorname*{cl}%
\nolimits^{w^{**}}\left(  (\partial_{\varepsilon}g)^{-1}(x^{\ast})+\mathrm{N}%
_{L\cap\operatorname*{dom}f}(x^{\ast})\right)  \right\} .
\]
If, in addition, $f$ is finite and (norm-) continuous somewhere, then
\begin{align*}
\partial f(x^{\ast})  &  =\overline{\operatorname*{co}}\left\{ 
(\partial\overline{g}^{w^{**}})^{-1}(x^{\ast}) \right\}  +\mathrm{N}%
_{\operatorname*{dom}f}(x^{\ast}).
\end{align*}
\end{prop}

\begin{proof}
Following similar arguments as those used in \cite{CH12}, we apply Proposition \ref{conj} in the duality pair $((X^{**},w^{**}),(X^*,\|\|_*))$, replacing the function $g$ in (\ref{eqs}) by the function $\hat{g}$ defined on $X^{**}$ as 
$$
\hat{g}(y)=g(y),\ \text{ if }y\in X^{**};\ +\infty, \text{ otherwise }.
$$
Observe that the $w^{**}$-lsc hull of $\hat{g}$ is precisely the function $\overline{g}^{w^{**}}$. 
\end{proof}

\bigskip

Now, as in \cite{CHL19,CHL19c}, we consider the following
convex semi-infinite optimization problem
\[
\mathcal{(P)}:\text{ \ }\operatorname*{Inf} f_{0}(x),\ \text{ subject to } f_{t}(x)\leq0,\text{ }t\in T,
\]
where $T$ is a given set, and\textbf{ }$f_{0},$ $f_{t}:\mathbb{R}%
^{n}\rightarrow\mathbb{R}_{\infty},$ $t\in T$\textbf{,} are proper and convex.
We assume, without loss of generality, that\textbf{ }$0\notin T$\textbf{, }and denote
\[
f:=\sup\nolimits_{t\in T}f_{t}.
\]

The following result establishes\ new Fritz-John and KKT optimality conditions
for problem $\mathcal{(P)}$, improving similar results in \cite{CHL19,CHL19c}.
Here we adopt the convention $\mathbb{R}_{+}\emptyset
=\left\{  0_{n}\right\} .$

\begin{prop}
Let $\bar{x}\ $be an optimal\ solution\ of $\mathcal{(P)}$ such that
$f(\bar{x})=0.$ Then\ we have
\begin{equation}
0_{n}\in\operatorname*{co}\left\{  \partial(f_{0}+\mathrm{I}%
_{\operatorname*{dom}f})(\bar{x})\cup\bigcap\nolimits_{\varepsilon
>0}\operatorname*{cl}\left(  \bigcup\nolimits_{t\in T_{\varepsilon}(\bar{x}%
)}\partial_{\varepsilon}(f_{t}+\mathrm{I}_{\operatorname*{dom}f\cap
\operatorname*{dom}f_{0}})(\bar{x})\right)  \right\}  . \label{FJC}%
\end{equation}
Moreover, if the Slater condition holds; that is, $f(x_{0})<0$ for some
$x_{0}\in\operatorname*{dom}f_{0}$, then
\begin{equation}
0_{n}\in\partial(f_{0}+\mathrm{I}_{\operatorname*{dom}f})(\bar{x}%
)+\operatorname*{cone}\bigcap\nolimits_{\varepsilon>0}\operatorname*{cl}%
\left(  \bigcup\nolimits_{t\in T_{\varepsilon}(\bar{x})}\partial_{\varepsilon
}(f_{t}+\mathrm{I}_{\operatorname*{dom}f\cap\operatorname*{dom}f_{0}})(\bar
{x})\right)  \label{Slater}%
\end{equation}
and, provided in addition that $f$ is continuous at some point in
$\operatorname*{dom}f_{0}\cap\operatorname*{dom}f,$
\begin{equation}
0_{n}\in\partial f_{0}(\bar{x})+\operatorname*{cone}\bigcap
\nolimits_{\varepsilon>0}\operatorname*{cl}\left(  \bigcup\nolimits_{t\in
T_{\varepsilon}(\bar{x})}\partial_{\varepsilon}f_{t}(\bar{x})\right)
+\mathrm{N}_{\operatorname*{dom}f}(\bar{x}). \label{Slaterc}%
\end{equation}

\end{prop}

\begin{proof}
We consider the function $g:\mathbb{R}^{n}\rightarrow\mathbb{R\cup\{+\infty
\}}$, defined as
\[
g(x):=\sup\{f_{0}(x)-f_{0}(\bar{x}),\ f_{t}(x),t\in T\}=\max\left\{
f_{0}(x)-f_{0}(\bar{x}),\ f(x)\right\}  ,
\]
so that $\operatorname*{dom}g=\operatorname*{dom}f_{0}\cap\operatorname*{dom}%
f.$ Then $\bar{x}$ is a global minimum of $g;$ that is, $0_{n}\in\partial
g(\bar{x}).$ To proceed, we first apply Proposition \ref{thmcompact0} to the
(finite) family $\{f_{0}-f_{0}(\bar{x}),\ f\}$ and obtain
\begin{equation}
0_{n}\in\operatorname*{co}\left\{  \partial(f_{0}+\mathrm{I}%
_{\operatorname*{dom}f})(\bar{x})\cup\partial(f+\mathrm{I}%
_{\operatorname*{dom}f_{0}})(\bar{x})\right\}  . \label{mous}%
\end{equation}
But Theorem \ref{hamdkathirbisb}, applied to the family $\left\{
f_{t}+\mathrm{I}_{\operatorname*{dom}f_{0}},\text{ }t\in T\right\}  $, yields%
\begin{equation}
\partial(f+\mathrm{I}_{\operatorname*{dom}f_{0}})(\bar{x})=\operatorname*{co}%
\left\{  \bigcap\nolimits_{\varepsilon>0}\operatorname*{cl}\left(
\bigcup\nolimits_{t\in T_{\varepsilon}(\bar{x})}\partial_{\varepsilon}%
(f_{t}+\mathrm{I}_{\operatorname*{dom}f\cap\operatorname*{dom}f_{0}})(\bar
{x})\right)  \right\}  , \label{ss}%
\end{equation}
and (\ref{FJC}) follows from (\ref{mous}).

Finally,\ it can be easily seen from (\ref{mous}) that the Slater condition
precludes that $0_{n}\in\partial(f+\mathrm{I}_{\operatorname*{dom}f_{0}}%
)(\bar{x})$. So, (\ref{Slater}) follows from (\ref{FJC}). Under the
supplementary continuity condition, Corollary \ref{khay} ensures that%
\begin{align*}
\partial(f+\mathrm{I}_{\operatorname*{dom}f_{0}})(\bar{x})  &  =\mathrm{N}%
_{\operatorname*{dom}f_{0}}(\bar{x})+\partial f(\bar{x})\\
&  =\mathrm{N}_{\operatorname*{dom}f_{0}}(\bar{x})+\mathrm{N}%
_{\operatorname*{dom}f}(x)+\operatorname*{co}\left\{  \bigcap
\nolimits_{\varepsilon>0}\operatorname*{cl}\left(  \bigcup\nolimits_{t\in
T_{\varepsilon}(x)}\partial_{\varepsilon}f_{t}(x)\right)  \right\}  ,
\end{align*}
and (\ref{Slaterc}) follows, taking into account (\ref{MR}) and 
\begin{align*}
0_{n}  &  \in\partial(f_{0}+\mathrm{I}_{\operatorname*{dom}f})(\bar
{x})+\mathbb{R}_{+}\partial(f+\mathrm{I}_{\operatorname*{dom}f_{0}})(\bar
{x})\\
&  =\partial f_{0}(\bar{x})+\mathrm{N}_{\operatorname*{dom}f}(\bar
{x})+\mathbb{R}_{+}\partial(f+\mathrm{I}_{\operatorname*{dom}f_{0}})(\bar
{x})\\
&  =\partial f_{0}(\bar{x})+\operatorname*{cone}\left\{  \bigcap
\nolimits_{\varepsilon>0}\operatorname*{cl}\left(  \bigcup\nolimits_{t\in
T_{\varepsilon}(x)}\partial_{\varepsilon}f_{t}(x)\right)  \right\}
+\mathrm{N}_{\operatorname*{dom}f_{0}}(\bar{x})+\mathrm{N}%
_{\operatorname*{dom}f}(x)\\
&  \subset\partial f_{0}(\bar{x})+\operatorname*{cone}\left\{  \bigcap
\nolimits_{\varepsilon>0}\operatorname*{cl}\left(  \bigcup\nolimits_{t\in
T_{\varepsilon}(x)}\partial_{\varepsilon}f_{t}(x)\right)  \right\}
+\mathrm{N}_{\operatorname*{dom}f}(x).
\end{align*}

\end{proof}

\section{Conclusions}
The main conclusion of this work is that the compactification method proposed in the paper allows us to move from the non-continuous setting to the continuous one and the other way around, as well as to develop a unifying theory which inspires new results and applications.
The main results in relation to the subdifferential of the supremum are stated in Theorems \ref{hamdkathir}, \ref{thm6}, and \ref{hamdkathirbisb}, which are established in the most general framework, free of assumptions on the index set and the data functions. Our results cover most of the existing formulas such as those obtained in \cite{CHL16,CHL19,CHL19b,CHL19c,DGL06,HL08,HLZ08,HiMa93,Io12,IoLe72,IoTi79,LoTh13,LV10,Ps65,So01,Va69,Vo94}. The Fritz-John and KKT conditions for convex semi-infinite optimization are expressed in the most general scenario and, consequently, extend some previous results which can be found in \cite{DGL06,GL98,HiMa93,IoTi79}.

\end{document}